\documentclass[reqno,11pt]{article}

\usepackage{a4wide,eucal,enumerate,mathrsfs,xspace}
\usepackage{amsmath,amssymb,epsfig,bbm}
\usepackage[toc,page]{appendix}
\usepackage[usenames]{color}
\usepackage{ifthen}

\numberwithin{equation}{section}

\newtheorem{theorem}{Theorem}[section]

\newtheorem{corollary}[theorem]{Corollary}
\newtheorem{lemma}[theorem]{Lemma}
\newtheorem{proposition}[theorem]{Proposition}
\newtheorem{definition}[theorem]{Definition}

\newtheorem{remark}[theorem]{Remark}

\usepackage{float}
  
  \usepackage{tikz}
\usetikzlibrary{patterns}
\usetikzlibrary{decorations.markings}

\tikzset{
  set arrow inside/.code={\pgfqkeys{/tikz/arrow inside}{#1}},
  set arrow inside={end/.initial=>, opt/.initial=},
  /pgf/decoration/Mark/.style={
    mark/.expanded=at position #1 with
    {
      \noexpand\arrow[\pgfkeysvalueof{/tikz/arrow inside/opt}]{\pgfkeysvalueof{/tikz/arrow inside/end}}
    }
  },
  arrow inside/.style 2 args={
    set arrow inside={#1},
    postaction={
      decorate,decoration={
        markings,Mark/.list={#2}
      }
    }
  },
}

\usepackage[normalem]{ulem}
\usepackage{pdfsync}
\usepackage{hyperref}

\newcommand{\N}{\mathbb{N}}

\newcommand{\R}{\mathbb{R}}

\newcommand{\Z}{\mathbb{Z}}

\renewcommand{\SS}{\mathscr{S}}

\newcommand{\cE}{{\ensuremath{\mathcal E}}}

\newcommand{\cI}{{\ensuremath{\mathcal I}}}

\newcommand{\cP}{{\ensuremath{\mathcal P}}}

\newcommand{\cR}{{\ensuremath{\mathcal R}}}

\newcommand{\cY}{{\ensuremath{\mathcal Y}}}

\newcommand{\sfc}{{\sf c}}
\newcommand{\sfd}{{\sf d}}

\newcommand{\sfi}{{\sf i}}
\newcommand{\sfl}{{\sf l}}

\newcommand{\sfr}{{\sf r}}
\newcommand{\sfs}{{\sf s}}

\newcommand{\sfD}{{\sf D}}

\newcommand{\frH}{{\mathfrak H}}

\newcommand{\frP}{{\mathfrak P}}

\newcommand{\rmd}{{\mathrm d}}

\newcommand{\rmC}{{\mathrm C}}
\newcommand{\rmD}{{\mathrm D}}
\newcommand{\rmE}{{\mathrm E}}

\newcommand{\rmL}{{\mathrm L}}
\newcommand{\rmM}{{\mathrm M}}

\newcommand{\rmS}{{\mathrm S}}

\newcommand{\Kliminf}{K\kern-3pt-\kern-2pt\mathop{\rm
lim\,inf}\limits}  
\newcommand{\Klimsup}{K\kern-3pt-\kern-2pt\mathop{\rm lim\,sup}\limits}  

\newcommand{\argmin}{\mathop{\rm argmin}\limits}   

\newcommand{\restr}[1]{\lower3pt\hbox{$|_{#1}$}}

\newcommand{\la}{{\langle}}                  
\newcommand{\ra}{{\rangle}}
\newcommand{\up}{\uparrow}

\newcommand{\eps}{\varepsilon}  
\newcommand{\nchi}{{\raise.3ex\hbox{$\chi$}}}

\def\qed{\ifmmode 
  \else \leavevmode\unskip\penalty9999 \hbox{}\nobreak\hfill
  \fi               
    \qquad           \hbox{\hskip.5em $\square$
                \hskip.1em}}
\def\endproofsym{\qed}

\newenvironment{proof}[1][Proof]{\def\endproofsym{\qed}\trivlist\item[\hskip\labelsep{%
\noindent{\normalfont\emph{#1}.}\hskip .321429\parindent}]\ignorespaces}
{\endproofsym\endtrivlist}

\newcommand{\leftl}{-}
\newcommand{\rightl}{+}

\newcommand{\equi}[1]{\sim_{\sfc}}

\newcommand{\tlim}{\tau\text{-}\kern-3pt\mathop{\rm lim}\limits}  

\newcommand{\fhaus}[2]{\mathscr H_{#1}^1
  \ifthenelse{\equal{#2}{}}{}{(#2)}}

\newcommand{\Jmp}[1]{\mathrm{Jmp}_{#1}}

\newcommand{\mVar}[1]{\text{\sl{Var}}_{#1}}
\newcommand{\VE}{{\rm VE}}
\newcommand{\BV}{\mathrm{BV}}

\DeclareMathOperator{\varpsi}{Var_\Psi}

\DeclareMathOperator{\Ju}{J_u}

\DeclareMathOperator{\Res}{\cR}
\DeclareMathOperator{\Cf}{Trc} 
\DeclareMathOperator{\Fd}{\mathsf c}
\DeclareMathOperator{\Cd}{\mathrm{GapVar}_\delta}

\DeclareMathOperator{\cost}{\sfc}

\newcommand{\rmJ}{\mathrm J}

\newcommand{\SSd}{\SS\kern-3pt_\sfd}
\newcommand{\SSD}{\SS_\sfD}

\newcommand{\Holes}{\mathfrak H}
\newcommand{\Pf}{\mathfrak P_f}
\newcommand{\mytag}[2]{\hyperref[#1#2]{$\la$#1\ifthenelse{\equal{#2}{}}{}{.#2}$\ra$}}

\newcommand{\Wir}{\mathit W_{\mathsf{i}\mathsf{r}}'}
\newcommand{\Wsl}{\mathit W_{\mathsf{s}\mathsf{l}}'}
\newcommand{\Wird}[1]{\mathit W_{\mathsf{i}\mathsf{r},{#1}}'}
\newcommand{\Wsld}[1]{\mathit W_{\mathsf{s}\mathsf{l},{#1}}'}

\newcommand{\maxW}{\textit{\textbf{m}}^{\bar{u}}_\delta}
\newcommand{\invmaxW}{\textit{\textbf{p}}^{\bar{u}}_\delta}
\newcommand{\minW}{\textit{\textbf{n}}^{\bar{u}}_\delta}
\newcommand{\invminW}{\textit{\textbf{q}}^{\bar{u}}_\delta}
\newcommand{\invmaxWc}{\textit{\textbf{p}}^{\bar{u}}_{\sfc,\delta}}
\newcommand{\invmaxWca}{\textit{\textbf{p}}^{u(a)}_{\sfc,\delta}}
\newcommand{\invminWc}{\textit{\textbf{q}}^{\bar{u}}_{\sfc,\delta}}
\newcommand{\invminWca}{\textit{\textbf{q}}^{u(a)}_{\sfc,\delta}}
\DeclareMathOperator{\ul}{\mathit{u}_{\sfl}}
\DeclareMathOperator{\ur}{\mathit{u}_{\sfr}}

\title{Visco-Energetic solutions to one-dimensional \\ rate-independent problems}
\begin{document}

\author{Luca Minotti
\thanks{Universit\`a di Pavia. 
email:
\textsf{luca.minotti01@universitadipavia.it}. 
}
}

\maketitle

\begin{abstract} 
Visco-Energetic solutions of rate-independent systems (recently introduced in \cite{MinSav16}) are obtained by solving a modified time Incremental Minimization Scheme, where at each step the dissipation  is reinforced by a viscous correction $\delta$, typically a quadratic perturbation of the dissipation distance. Like Energetic and Balanced Viscosity solutions, they provide a variational characterization of rate-independent evolutions, with an accurate description of their jump behaviour.  
\par
In the present paper we study Visco-Energetic solutions in the one-dimensional case and we obtain a full characterization for a broad class of energy functionals. In particular, we prove that they exhibit a sort of intermediate behaviour between Energetic and Balanced Viscosity solutions, which can be finely tuned according to the choice of the viscous correction $\delta$.
\end{abstract}

\tableofcontents

\section{Introduction} Rate-independent problems occur in several contexts. We refer the reader to the recent monograph \cite{Mielke-Roubicek15} for a survey of rate-independent modeling and analysis in a wide variety of applications.
The analytical theory of rate-independent evolutions encounters some mathematical challenges, which are apparent even in the simplest example, the \emph{doubly nonlinear} differential inclusion
\begin{equation} \label{DN}
\partial \Psi(u'(t))+\rmD \cE(t,u(t))\ni 0\quad \text{in $X^*$}\quad\text{ for a.a. $t\in(a,b)$}. \tag{DN}
\end{equation}
Here $X^*$ is the dual of a finite-dimensional linear space, $\rmD\cE$ is the (space) differential of a time-dependent energy functional $\cE\in \rmC^1([a,b]\times X;\R)$ and $\Psi: X\rightarrow (0,+\infty)$ is a convex and nondegenerate dissipation potential, hereafter supposed \emph{positively homogeoneous of degree 1}.
\par
It is well known that if the energy $\cE(t,\cdot)$ is not strictly convex, one cannot expect the existence of an absolutely continuous solution to \eqref{DN}, so that the natural space for candidate solutions $u$ is $\BV([a,b];X)$. This fact has motivated the development of various weak formulations of \eqref{DN}, which should also take into account the behaviour of $u$ at jump points.

\paragraph{Energetic solutions.}  The first is the notion of \emph{Energetic solutions}, \cite{Mielke-Theil-Levitas02,Mielke-Theil04,Mainik-Mielke05}. For the simplified rate-independent evolution \eqref{DN}, Energetic solutions are curves $u:[a,b]\to X$ with bounded variation that are characterized by two variational conditions, called \emph{stability} (S$_\Psi$) and \emph{energy balance} (E$_\Psi$):
\begin{equation} \label{en-stability}
\cE(t,u(t))\le \cE(t,z)+\Psi(z-u(t))\qquad\text{for every $z\in X$}, \tag{$\rmS_\Psi$}
\end{equation}
\begin{equation} \label{en-energy-balance}
\cE(t,u(t))+\varpsi(u;[a,t])=\cE(a,u(a))+\int_a^t\partial_t\cE(s,u(s))\,\rmd s, \tag{$\rmE_\Psi$}
\end{equation}
where $\varpsi$ is the pointwise total variation with respect to $\Psi$ (see \eqref{eq:total-variation} in section \ref{sec:2} for the precise definition).
\par
One of the strongest feature of the energetic approach is the possibility to construct energetic solutions by solving the \emph{time Incremental Minimization Scheme}
\begin{equation} \label{IMS}
\min_{U\in X} \cE(t_\tau^n,U)+\Psi\big(U-U_\tau^{n-1}\big). \tag{$\mathrm{IM}_\Psi$}
\end{equation}
If $\cE$ has compact sublevels then for every ordered partition 
$\tau=\{t^0_\tau=a,t^1_\tau,\cdots,t^{N-1}_\tau,t^N_\tau=b\}$ of the
interval $[a,b]$ with variable time step
$\tau^n:=t^n_\tau-t^{n-1}_\tau$ and for every
initial choice $U^0_\tau= u(a)$
we can construct by induction an approximate
sequence $(U^n_\tau)_{n=0}^N$ solving \eqref{IMS}. If $\overline U_\tau$ denotes the left-continuous piecewise constant
interpolant
of $(U^n_\tau)_n$, then the family of discrete solutions $\overline
U_\tau$
has limit curves with respect to pointwise convergence as 
the maximum of the step sizes
$|\tau|=\max \tau^n$ vanishes, and every limit curve $u$ is an energetic solution.

\par
Consider for instance the 1-dimensional example when the energy has the form
\begin{equation} \label{1d-energy}
\cE(t,u):=W(u)-\ell(t)u\quad\text{for a double-well potential such as $W(u)=(u^2-1)^2$}.
\end{equation}
When the loading $\ell \in \rmC^1([a,b])$ is strictly increasing, $\Psi(v):=\alpha |v|$ with $\alpha>0$, and $u(a)$ is choosen carefully, it is possible to prove, \cite{Rossi-Savare13}, that an Energetic solution $u$ is an increasing selection of the equation
\begin{equation} \label{eq:23}
\alpha+W^{**}(u(t))\ni\ell(t)\qquad\text{for every $t\in [a,b]$},
\end{equation}
where $W^{**}$ is the convex envelope $W^{**}(u)=((u^2-1)_+)^2$.

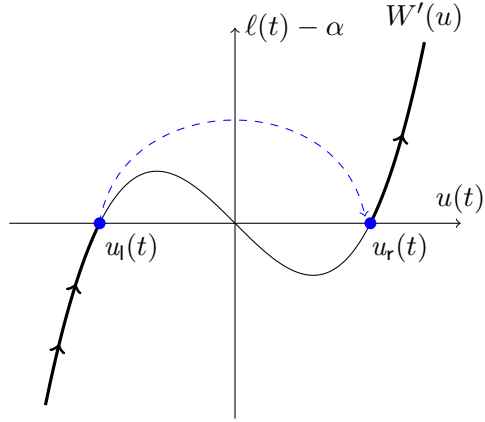
\begin{figure}[!ht]
\centering
\begin{tikzpicture}

\draw[->] (-3,0) -- (3,0) node[above] {$u(t)$};
\draw[->] (0,-2.6) -- (0,2.6) node[right] {$\ell(t)-\alpha$};

 \begin{scope}[scale=1.8]
    \draw[domain=-1.4:1.4,samples=100] plot ({\x}, {(\x)^3-\x}) node[above] {$W'(u)$};
     \draw[domain=-1.4:-1,very thick,samples=100] plot ({\x}, {(\x)^3-\x}) [arrow inside={}{0.33,0.66}];
     \draw[domain=1:1.4,very thick,samples=100] plot ({\x}, {(\x)^3-\x}) [arrow inside={}{0.50}]    ;
      \foreach \Point in {(-1,0) ,(1,0) }{
    \draw[fill=blue,blue] \Point circle(0.04);
    }
\draw[->, dashed,blue] (-1,0) to [bend left=80, looseness=1.3] (0.96,0.05);
\node[below] at (-1,0) {\qquad\,$\ul(t)$};
\node[below] at (1,0) {\qquad $\ur(t)$};
  \end{scope}
  
\end{tikzpicture} 
\caption{Energetic solution for a double-well energy $W$ with an
  increasing load $\ell$.}
\label{fig:1}
\end{figure}
In this context, the solution $u$ have a jump when it is satisfied the so-called \emph{Maxwell rule}
\begin{equation}
\int_{\ul(t)}^{\ur(t)}\Big( W'(r)-\ell(t)+\alpha_+ \Big)\,\rmd r=0.
\end{equation}
The latter evolution mode prescribes that for all $t\in [a,b]$, the function $u(t)$ only attains \emph{absolute minima} of the function $u\mapsto W(u)-(\ell(t)-\alpha_+)u$. This corresponds to a convexification of $W$ and causes the system to jump ``early''.

\paragraph{Balanced Viscosity (BV) solutions.} The global stability condition \eqref{en-stability} may lead the system to change instantaneously in a very drastic way, jumping into far apart energetic configurations. In order to obtain a formulation where local effects are more relevant (see 
\cite{DalMaso-Toader02, NegOrt07?QSCP,Efendiev-Mielke06}), a natural idea is to consider rate-independent evolution as the limit of systems with smaller and smaller viscosity, namely to study the approximation of \eqref{DN}
\begin{equation} \label{eq:DNepsilon}
\partial\Psi_\varepsilon(u'(t))+\rmD\cE(t,u(t))\ni 0 \quad\text{in $X^*$},\qquad \Psi_\varepsilon(v):=\Psi(v)+\frac{\varepsilon}{2}\Psi^2(v), \tag{$\mathrm{DN}_\varepsilon$}
\end{equation}
which corresponds to introduce a quadratic (or even more general) perturbation in the time Incremental Minimization Scheme:
\begin{equation} \label{eq:173}
\min_{U\in X} \cE(t_\tau^n,U)+\Psi\big(U-U_\tau^{n-1}\big)+\frac{\varepsilon^n}{2\tau^n}\Psi^2\big(U-U_\tau^{n-1}\big). \tag{$\mathrm{IM}_{\Psi,\varepsilon}$}
\end{equation}
The choice $\varepsilon^n=\varepsilon^n(\tau)\downarrow 0$ with $ \frac{\varepsilon^n(\tau)}{|\tau|}\uparrow+\infty$ leads to the notion of \emph{Balanced Viscosity} solutions \cite{Rossi-Mielke-Savare08,Mielke-Rossi-Savare09,Mielke-Rossi-Savare12,Mielke-Rossi-Savare13}.
Under suitable smoothness and lower semicontinuity assumptions, it is possible to prove that 
all the limit curves satisfy a \emph{local stability} condition and a
modified energy balance, involving 
an augmented total variation that 
encodes a more refined description of the jump behaviour of $u$:
roughly speaking, a jump between $\ul(t)$ and $\ur(t)$ occurs only when
these values can be connected by a rescaled solution $\vartheta$ of
\eqref{eq:DNepsilon}, where the energy is frozen at the jump time $t$
\begin{equation}
  \label{eq:173bis}
  \partial\Psi(\vartheta'(s))+\vartheta'(s)+\rmD\cE(t,\vartheta(s))\ni0.
\end{equation}

In the one-dimensional example \eqref{1d-energy}, with the loading $\ell$ strictly increasing and under suitable choices of the initial datum, it is possible to  prove, \cite{Rossi-Savare13}, that $u$ is a BV solution if and only if it is nondecreasing and
\begin{equation} \label{eq:24}
\alpha+W'(u(t))=\ell(t)\qquad \text{for all $t\in[a,b]\setminus \Ju$}.
\end{equation}
\begin{figure}[!ht]
  \centering
  \begin{tikzpicture}

\draw[->] (-3,0) -- (3,0) node[above] {$u(t)$};
\draw[->] (0,-2.8) -- (0,2.8) node[right] {$\ell(t)-\alpha$};

 \begin{scope}[scale=1.8]
    \draw[domain=-1.4:1.4,samples=100] plot ({\x}, {(\x)^3-\x}) node[above] {$W'(u)$};
    \draw[domain=-1.4:-0.57,very thick,samples=100] plot ({\x}, {(\x)^3-\x}) [arrow inside={}{0.33,0.66}];
     \draw[domain=1.15:1.4,very thick,samples=100] plot ({\x}, {(\x)^3-\x}) [arrow inside={}{0.50}]    ;
     \draw[domain=-0.57:1.15,thick,blue,samples=100] plot ({\x}, {0.385}) [arrow inside={}{0.33,0.66}];
     \foreach \Point in {(-0.57,0.385),(1.15,0.385)}{
    \draw[fill=blue,blue] \Point circle(0.04);
}
\draw[dotted] (-0.57,0.385) -- (-0.57,0) node[below] {$\ul(t)$};
\draw[dotted] (1.15,0.385) -- (1.15,0) node[below] {\,\,\,$\ur(t)$};
 \end{scope}
  
\end{tikzpicture}

  \caption{BV solution for a double-well energy $W$ with an
  increasing load $\ell$. The blue line denotes the 
  path described by the optimal transition $\vartheta$ solving 
\eqref{eq:173bis}.}
  \label{fig:2}
\end{figure}
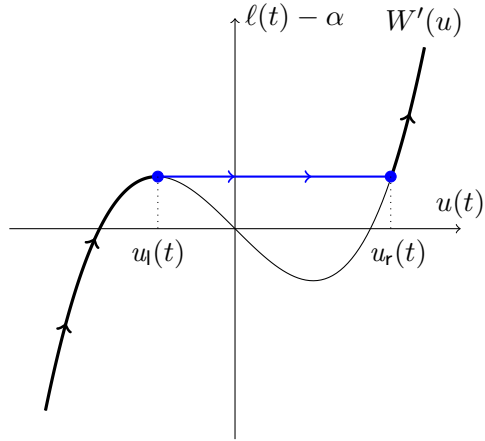

The evolution mode \eqref{eq:24} follows the so called \emph{Delay rule}, related to hysteresis behaviour. The system accepts also \emph{relative minima} of $u\mapsto W(u)-(\ell(t)-\alpha)u$, and thus the function $t\mapsto u(t)$ tend to jump ``as late as possible''.

\paragraph{Visco-Energetic solutions and main results of the paper.} Recently, in \cite{MinSav16}, the new notion of Visco-Energetic (VE) solutions has been proposed. This is a sort of intermediate situation between energetic and balanced viscosity, since these solutions are obtained by studying the time Incremental Minimization Scheme \eqref{eq:173} when one keeps constant the ratio
$\mu:=\eps^n/\tau^n$. In this way 
the dissipation $\Psi$ is corrected by an extra viscous
penalization term, for example of the form
\begin{equation} \label{eq:169}
\delta(u,v):=\frac{\mu}2\Psi^2(v-u)\qquad \text{for every $u,v\in X,\quad\mu\ge0$,}
\end{equation}
which induces a stronger localization  of the minimizers, according to the size of the parameter $\mu$. The new modified time Incremental Minimization Scheme is therefore
\begin{equation}
  \label{eq:167mu}
  \min_{U\in  X} \cE(t_\tau^n,U)+\Psi\big(U-U_\tau^{n-1}\big)+\delta(U,U_\tau^{n-1}) \tag{$\mathrm{IM}_{\Psi,\mu}$}.
\end{equation}
\par
As in the energetic and BV cases, a variational characterization of the functions obtained as a limit of the solution of \eqref{eq:167mu} is possible, still involving a suitable stability condition and an energetic balance. Concerning stability, we have a natural generalization of \eqref{en-stability}:
\begin{equation} \label{eq:168}
  \cE(t,u(t))\le \cE(t,v)+\Psi(v-u(t))+\delta(u(t),v) \quad\text{for every }v\in X,\
  t\in [a,b]\setminus \Ju.
  \tag{S$_\sfD$}
\end{equation}
The right replacement of the energy balance condition is harder to formulate. A heuristic idea, which one can figure out by the direct analysis of \eqref{1d-energy}, is that jump transitions between $\ul(t)$ and $\ur(t)$ should be described by discrete trajectories $\vartheta:Z\to X$ defined in a subset $Z\subset \Z$ 
such that each value $\vartheta(n)$ is a minimizer of the incremental problem \eqref{eq:167mu}, with datum $\vartheta(n-1)$ and with the energy ``frozen'' at time $t$.
In the simplest cases $Z=\Z$, the left and right jump values
 are the limit of $\vartheta(n)$ as $n\to\pm\infty$, but more
complicated situations can occur, when $Z$ is a proper subset of $\Z$
or one has to deal with concatenation of (even countable) discrete
transitions and sliding parts parametrized by a continuous variable, where the stability condition \eqref{eq:168} holds. 
\par
In order to capture all of these possibilities, VE transitions are parametrized by continuous maps $\vartheta:E\to X$ defined
in an \emph{arbitrary compact subset} of $\R$. We refer to section \ref{subsec:Visco-Energetic-Euclidean} for the precise description of the new dissipation cost and the corresponding total variation. 
\par
In the present paper we study Visco-Energetic solutions in the one dimensional setting and we obtain a full characterization for the same broad class of energy functionals of \cite{Rossi-Savare13}. Respect to Energetic and BV solutions, the main difficulty here comes from the description of solutions at jumps: as we have mentioned, transitions are now defined in an arbitrary compact subset of $\R$, so that a wide range of possibilities can occur. For instance, the energetic case is a very particular situation, where (e.g. for an increasing jump) the transitions have the form
\begin{equation}
\vartheta: \{0;1\}\rightarrow\R\qquad\text{such that $\vartheta(0)=\ul(t),\quad\vartheta(1)=\ur(t)$},
\end{equation}
defined in a compact set that consists just in two points.
\par
However, thanks to an accurate analysis of VE dissipation cost, we are able to describe all these possibilities. Coming back to the standard example \eqref{1d-energy}, with the viscous correction $\delta$ of the form \eqref{eq:169}, the behaviour of VE solutions strongly depends on the parameter $\mu$. More precisely, the following situations can occur:
\begin{itemize}
\item The viscous correction term is ``strong'', for example $\mu \ge -\min W''$. In this case VE solutions exhibits a behaviour comparable to BV solutions: both satisfies the same local stability condition and equation \eqref{eq:24} holds, so that they follow a \emph{delay rule}.
\item  No viscous corrections are added to the system, which corresponds to $\mu=0$. In this case VE solutions coincides with energetic solutions, equation \eqref{eq:23} holds and they satisfy the \emph{Maxwell rule}.
\item A ``weak'' viscous correction is added to the system, which corresponds to a small $\mu>0$. We have a sort of intermediate situation between the two previous cases: a jump can occur even before reaching a local extremum of $W'$. In particular, an increasing jump can occur when the \emph{modified Maxwell rule} is satisfied:
\begin{equation} \label{eq:201}
\int_{\ul(t)}^{u_+}\Big( W'(r)-\ell(t)+\alpha+\mu (r-\ul(t))\Big)\,\rmd r=0,\qquad \text{for some $u_+>\ul(t)$}.
\end{equation}
In this case $\ur(t)$ may differ from $u_+$: see Figure \ref{fig:3} for more details.
\end{itemize}

\begin{figure}[!ht]
\centering
\begin{tikzpicture}

\draw[->] (-3,0) -- (3,0) node[above] {$u(t)$};
\draw[->] (0,-2.8) -- (0,2.8) node[right] {$\ell(t)-\alpha$};

 \begin{scope}[scale=2]
    \draw[domain=-1.4:1.4,samples=200] plot ({\x}, {(\x)^3-\x}) node[above] {$W'(u)$};
     \draw[domain=-1.4:-0.57,very thick,samples=200] plot ({\x}, {(\x)^3-\x}) [arrow inside={}{0.33,0.66}];
     \draw[domain=1.1547:1.4,very thick,samples=200] plot ({\x}, {(\x)^3-\x}) [arrow inside={}{0.50}]    ;
     \foreach \Point in {(-0.57,0.3849),(1.1547,0.3849)}{
       \draw[fill=blue,blue] \Point circle(0.04);
}
     \foreach \Point in {(-0.542,0.3849),(-0.536,0.3849), (-0.530,0.3849), (-0.524,0.3849), (-0.518,0.3849), (-0.512,0.3849), (-0.506,0.3849), (-0.500,0.3849), (-0.494,0.3849), (-0.487,0.3849), (-0.479,0.3849), (-0.469,0.3849),(-0.457,0.3849),(-0.442,0.3849),(-0.423,0.3849),(-0.398,0.3849),(-0.363,0.3849),(-0.311,0.3849),(-0.225,0.3849),(-0.065,0.3849),(0.241,0.3849),(0.624,0.3849),(0.901,0.3849),(1.045,0.3849),(1.109,0.3849),(1.136,0.3849),(1.147,0.3849),(1.151,0.3849),(1.153,0.3849)}{
    \draw[fill=blue,blue] \Point circle(0.01);
}
\draw[->, dashed,blue]  (-0.225,0.3849) to [bend left=90,looseness=1.5]
(-0.065,0.3849);
\draw[fill=blue,blue] (-0.065,0.3849) circle(0.02);
\draw[->, dashed,blue]  (-0.065,0.3849) to [bend left=80,looseness=2]
(0.241,0.3849);
\draw[fill=blue,blue] (0.241,0.3849) circle(0.03);
\draw[->, dashed,blue]  (0.241,0.3849) to [bend left=80,looseness=2.5]
(0.624,0.3849);
\draw[fill=blue,blue] (0.624,0.3849) circle(0.03);
\draw[->, dashed,blue]  (0.624,0.3849) to [bend left=80, looseness=2]
(0.901,0.3849);
\draw[fill=blue,blue] (0.901,0.3849) circle(0.03);
\draw[->, dashed,blue]  (0.901,0.3849) to [bend left=80, looseness=1.5]
(1.045,0.3849);
\draw[fill=blue,blue](1.045,0.3849) circle(0.02);

\draw[dotted] (-0.57,0.3849) -- (-0.57,0) node[below] {$\ul(t)$};
\draw[dotted] (1.1547,0.3849) -- (1.1547,0) node[below] {\,\,\,$\ur(t)$};
  \end{scope}
\end{tikzpicture}
\quad
\begin{tikzpicture}

\draw[->] (-3.5,0) -- (3.5,0) node[above] {$u(t)$};
\draw[->] (0,-3) -- (0,3) node[right] {$\ell(t)-\alpha$};

 \begin{scope}[scale=2]
    \draw[domain=-1.4:1.4,samples=200] plot ({\x}, {(\x)^3-\x}) node[above] {$W'(u)$};
     \draw[domain=-1.4:-0.816,very thick,samples=200] plot ({\x}, {(\x)^3-\x}) [arrow inside={}{0.33,0.66}];
     \draw[domain=1.1153:1.4,very thick,samples=200] plot ({\x}, {(\x)^3-\x}) [arrow inside={}{0.50}]    ;
     \foreach \Point in {(-0.816,0.272),(1.1153,0.272)}{
     \draw[fill=blue,blue] \Point circle(0.04);
     }
     \draw[fill=blue,blue] (0.816,0.272) circle(0.03);
     \foreach \Point in {(1.082,0.272), (1.112,0.272), (1.115,0.272)}{
    \draw[fill=blue,blue] \Point circle(0.01);
}
\draw[->, dashed,blue] (-0.816,0.272) to [bend left=60] (0.8,0.3);
\draw[dashed ,blue] (0.81,0.272) to [bend left=60,looseness=3] (0.92,0.3);
\draw[dashed,blue] (0.95,0.272) to [bend left=60,looseness=3] (1.05,0.3);

\draw[fill=blue,blue] (0.95,0.272) circle(0.03);
\draw[fill=blue,blue] (1.05,0.272) circle(0.02);

\draw[dotted] (-0.816,0.272) -- (-0.816,0) node[below] {\,\,\,$\ul(t)$};
\draw[dotted] (0.816,-0.272) -- (0.816,0) node[below] {$u_+\quad$};
\draw[dotted] (1.1153,0.272) -- (1.1153,0) node[below] {\quad $\ur(t)$};

\draw[red] (-0.816,0.272) -- (0.816,-0.272);
  \end{scope}
\end{tikzpicture}
\caption{Visco-Energetic solutions for a double-well energy $W$ with an
  increasing load $\ell$. When $\mu>-\min W''$ (first picture) the solution jumps when it reach the maximum of $W'$ and the transition is the  ``double chain'' obtained by solving the Incremental Minimization Scheme with frozen time $t$. When $\mu$ is small (second picture) the optimal transition $\vartheta$  makes a first jump connecting $\ul(t)$ with $u_+$ according to the
  modified Maxwell rule \eqref{eq:201}: $\ul(t)$ and $u_+$ corresponds to
  the intersection of $W'$ with the red line, whose slope is $-\mu$. }
\label{fig:3}
\end{figure}
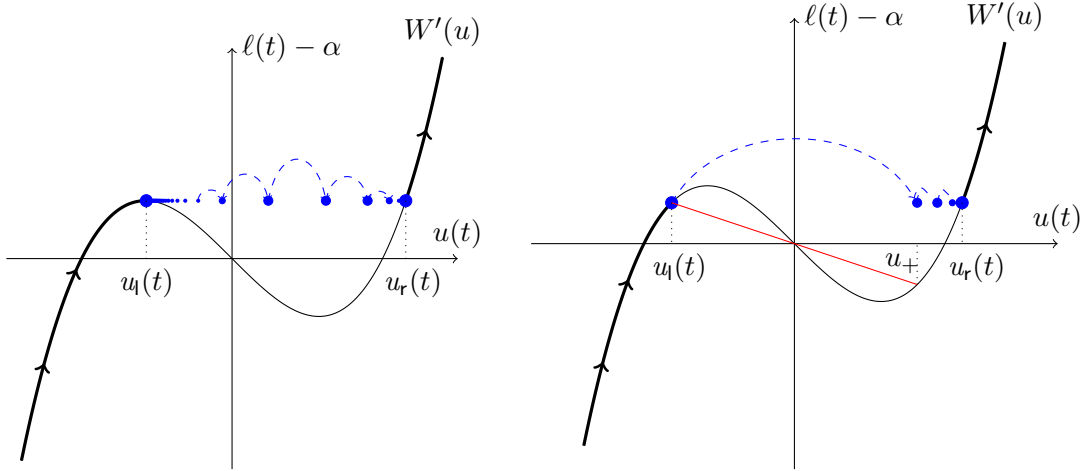

\paragraph{Plan of the paper.} In the paper we will analyse VE solutions to one-dimensional rate-independent evolutions driven by general (nonconvex) potentials and we will assume that the viscous corrections $\delta$ satisfies only the natural assumptions of the visco-energetic theory, including in particular the quadratic case \eqref{eq:169}.

In the preliminary section \ref{sec:2}, we recall the main definitions of Visco-Energetic solutions, their dissipation cost and the corresponding total variation, along with some useful properties and characterizations coming from the general theory; all the assumptions of the one-dimensional setting are collected in section \ref{subsec:one-dimensional-setting}.
\par
In section \ref{sec:3}, after a brief discussion about the stability conditions, we give a characterizations of Visco-Energetic solutions with a general (i.e. non monotone) external loading. This characterization involves the \emph{one-sided global slopes} with a $\delta$ correction, which are defined in section \ref{subsec:stability}.
\par
In section \ref{sec:4} we analyse the case of a monotone loading $\ell$. We exhibit a more explicit characterization of Visco-Energetic solutions, in term of the monotone envelopes of the one-sided global slopes. This characterization, in a suitable sense, generalizes \eqref{eq:23} and \eqref{eq:24}.

\section{Preliminaries} \label{sec:2} Throughout this section, $[a,b]\subseteq\R$ and
\[
(X,\|\cdot\|_X)\quad\text{ will be a finite dimensional normed vector space.}
\]
We first recall the key elements of the rate-independent system $(X,\cE,\Psi)$ along with the main definitions of Visco-Energetic solutions, their dissipation cost and some useful properties coming from the general theory, \cite{MinSav16}.

\subsection{Rate-independent setting and BV functions} \label{subsec:setting}
Hereafter we consider a rate-independent system $(X,\cE,\Psi)$, where the dissipation potential
\[
\Psi: X \rightarrow [0,+\infty)\text{ is 1-positively homogeneous, convex, with $\Psi(v)>0$ if $v\neq 0$,}
\]
and $\cE$ is a smooth, time dependent energy functional, which we take of the form
\begin{equation}
\cE(t,u):=W(u)-\langle\ell(t),u\rangle
\end{equation}
for some $W\in \rmC^1(X)$ bounded from below with a constant $-\lambda>-\infty$ and $\ell\in \rmC^1\left([a,b];X^*\right)$.  We shall also use the notation $\cP(t,u):=\partial_t\cE(t,u)=-\langle\ell'(t),u\rangle$ for the partial time derivative of $\cE$, and we set
\begin{equation} \label{eq:K*}
K^*:=\partial\Psi(0)=\{w\in X^*:\Psi_*(w)\le 1\}\subset X^*,\quad\text{where $\Psi_*(w):=\sup_{\Psi(v)\le 1}\langle w,v\rangle$}.
\end{equation}
\par
The rate-independent system associated with the energy functional $\cE$ and the dissipation potential $\Psi$ can be formally described by the \emph{rate-independent doubly nonlinear} differential inclusion
\begin{equation}
\partial \Psi(u'(t))+\rmD \cE(t,u(t))\ni 0\quad \text{in $X^*$}\quad\text{ for a.a. $t\in(a,b)$}. \tag{DN}
\end{equation}
\par
It is well known that for nonconvex energies, solutions to \eqref{DN} may exhibit discontinuities in time. Therefore, we shall consider functions of bounded variation pointwise defined in every $t\in[a,b]$, such that the pointwise total variation $\varpsi(u;[a,b])$ is finite, where
\begin{equation} \label{eq:total-variation}
\varpsi(u;[a,b]):=\sup\left\{\sum_{m=1}^M\Psi(u(t_m)-u(t_{m-1})):a=t_0<t_1<\dots<t_M=b\right\}.
\end{equation}
Notice that a function $u\in \BV([a,b];X)$ admits left and right limits at every $t\in [a,b]$:
\begin{equation}
\ul(t):=\lim_{s\uparrow t} u(s),\quad \ur(t):=\lim_{s\downarrow t} u(s),\quad\text{with $\ul(a):=u(a)$ and $\ur(b):=u(b)$}
\end{equation}
and its pointwise jump set $\Ju$ is the at most countable set defined by
\begin{equation}
\Ju:=\{t\in [a,b]: \ul(t)\neq u(t)\text{ or }u(t)\neq \ur(t)\}\supset\text{ess-$\Ju$}:=\{t\in (a,b):\ul(t)\neq\ur(t)\}.
\end{equation}
\par
We denote by $u'$ the distributional derivative of $u$ (extended by $u(a) \in (-\infty,a)$ and by $u(b)$ in $(b,+\infty)$): it is a Radon vector measure with finite total variation $|u'|$ supported in $[a,b]$. It is well known, \cite{Ambrosio-Fusco-Pallara00}, that $u'$ can be decomposed into the sum of its diffuse part $u'_{\mathrm{co}}$ and its jump part $u'_{\rmJ}$:
\[
u'=u'_{\mathrm{co}}+u'_{\rmJ},\quad u'\llcorner \text{ess-$\Ju$},\quad\text{so that $u'_{\mathrm{co}}(\{t\})=0$ for every $t\in[a,b]$.}
\]

\subsection{Visco-Energetic (VE) solutions in the finite-dimensional case}
\label{subsec:Visco-Energetic-Euclidean}
We recall the notion of Visco-Energetic solutions for the rate-independent system $(X,\cE,\Psi)$ introduced in section \ref{subsec:setting}. The first ingredient we need is a \emph{viscous correction}, namely a continuous map $\delta:X\times X\rightarrow [0,+\infty)$, and its associated augmented dissipation
\begin{equation}
\sfD(u,v):=\Psi(v-u)+\delta(u,v)\qquad \text{for every $u,v\in X$ }.
\end{equation} 
As in the energetic framework, \cite{Mielke-Theil-Levitas02,Mielke-Theil04,Mainik-Mielke05}, Visco-Energetic solutions to the rate-independent system $(X,\cE,\Psi)$ are curves $u:[a,b]\rightarrow X$ with bounded variation that are characterized by a \emph{stability condition} and an \emph{energetic balance}. 
\par
Concerning stability, we have a similar inequality, but we have to replace $\Psi$ with the augmented dissipation $\sfD$. More precisely, we will require that for every  $t\notin\Ju$ 
\begin{equation} \label{stability}
\cE(t,u(t))\le \cE(t,v)+\sfD(u(t),v)\quad\text{for every $v\in X$}, \tag{$\rmS_\sfD$}
\end{equation}
which is naturally associated with the $\sfD$ stable set $\SSD$.
\begin{definition}[\sfD-stable set] The $\sfD$-stable set is the subsets of $[a,b]\times X$ 
\begin{equation} \label{eq:SSD}
\SSD:=\left\{(t,u): \cE(t,u)\le \cE(t,v)+\sfD(u,v)\quad\text{for every $v\in X$}\right\}.
\end{equation}
Its section at time $t$ will be denoted with $\SSD(t)$.
\end{definition}
\par
As intuition suggests, not every viscous correction $\delta$ will be admissible for our purpose. A full description of Visco-Energetic solutions and admissible viscous corrections is discussed in \cite{MinSav16}, where the general metric-topological setting is considered. For the sake of simplicity, in this section we will assume that $\delta$ satisfies the following condition
\begin{equation} 
\lim_{v\rightarrow u}\frac{\delta(u,v)}{\Psi(v-u)}=0 \quad\text{for every $u \in  \SSD(t)$,\quad $t\in [a,b].$}
\end{equation}
\par
The \emph{energetic balance} is harder to formulate than stability and we first need to introduce the key concepts of transition cost and augmented total variation associated with the dissipation $\sfD$. 
\par
Hereafter, for every subset $E\subset\R$ we call $E^-:=\inf E$, $E^+:=\sup E$; whenever $E$ is compact, we will denote by $\frH(E)$ the (at most) countable collection of the connected components of the open set $[E^-,E^+]\setminus E$. We also denote by $\frP_f(E)$ the collection of all finite subsets of $E$. 
\par
Concerning the transition cost, the main point is to consider transitions parametrized by continuous maps $\vartheta:E\to X$ defined in arbitrary compact subsets of $\R$ such that $\vartheta(E^-)=\ul(t)$ and $\vartheta(E^+)=\ur(t)$. More precisely, the first ingredient will be a \emph{residual stability function}:
\begin{definition}[Residual stability function]
\label{def:res-stability}
 For every $t\in[a,b]$ and $u\in
  X$ 
  the residual stability function is defined by
  \begin{align}
    \label{eq:108}
    \Res(t,u):&=\sup_{v\in X}
               \{\mathcal{E}(t,u)-\mathcal{E}(t,v)-\sfD(u,v)\} \\
    \label{eq:109}
&=\mathcal{E}(t,u)-\inf_{v\in X}\{\mathcal{E}(t,v)+\sfD(u,v)\}.
\end{align}
\end{definition}
$\Res$ provides a measure of the 
failure of the stability condition \eqref{stability}, since 
for every $u\in X$, $t\in[a,b]$ we get 
\begin{equation}
  \label{eq:62}
  \cE(t,u)\le \cE(t,v)+\sfD(u,v)+\Res(t,u)
\end{equation}
and
\begin{equation}
\Res(t,u)=0\quad \Longleftrightarrow \quad \relax u\in
\SSD(t).\label{eq:63}
\end{equation}

The transition cost is the sum of three contributions, accordingly with the following definition.
\begin{definition}[Transition cost]
  \label{def:transition-cost}
  Let $E\subset \R$ compact and $\vartheta\in \rmC(E;X)$. For every
  $t\in[a,b]$ we define the \emph{transition cost function} $\Cf(t,\vartheta,E)$ by
\begin{equation}
  \label{eq:33}
  \Cf(t,\vartheta,E):=\varpsi(\vartheta,E)+\Cd(\vartheta,E)+\sum_{s\in E\setminus \{E^+ \}}\Res(t,\vartheta(s))
\end{equation}
where the first term is the usual total variation \eqref{eq:total-variation}, 
the second one 
is
\[
\Cd(\vartheta,E):=\sum_{I\in \Holes(E)}\delta(\vartheta(I^- ),\vartheta(I^+ )),
\]
and
the third term is
\[
\sum_{s\in
  E\setminus\{E^+ \}}\Res(t,\vartheta(s)):=\sup\left\{\sum_{s\in
    P}\Res(t,\vartheta(s)): P\in \Pf(E\setminus \{E^+ \})\right\},
\]
with the sum defined as $0$ if $E\setminus \{E^+ \}=\emptyset$.
\end{definition}

We adopt the convention $\Cf(t,\vartheta,\emptyset):=0$. 
It is not difficult to check that the transition cost
$\Cf(t,\vartheta,E)$ is additive with respect to $E$:
\begin{equation}
  \label{eq:195}
  \Cf(t,\vartheta,E\cap[a,c])=
  \Cf(t,\vartheta,E\cap[a,b])+
  \Cf(t,\vartheta,E\cap[b,c])\quad \text{for every }a<b<c.
\end{equation}
It has been proved, \cite[Theorem 6.3]{MinSav16}, that for every $t\in[a,b]$ and for every $\vartheta\in \rmC(E;X)$ 
\begin{equation}\label{eq:crinqualitytheta}
  \cE(t,\vartheta(E^+ ))+\Cf(t,\vartheta,E)\geq \cE(t,\vartheta(E^- )).
\end{equation}
The dissipation cost $\Fd(t,u_0,u_1)$ induced by the function $\Cf$ is
defined by minimizing $\Cf(t,\vartheta,E)$ among all the transitions $\vartheta$
connecting $u_0$ to $u_1$:
\begin{definition}[Jump dissipation cost and augmented total variation] \label{dissipationcost} 
Let $t\in[a,b]$ be fixed and let us consider $u_0, u_1\in X$. We set 
\begin{equation} \label{eq:dissipationcost}
\Fd(t,u_0,u_1):=\inf\left\{\Cf(t, \vartheta,E): E\Subset \R,\ 
\vartheta\in \rmC(E; X),\  \vartheta(E^- )=u_0,\ \vartheta(E^+ )=u_1\right\},
\end{equation}
with the incremental dissipation cost $\Delta_\sfc(t,
u_0,u_1):=\Fd(t,u_0,u_1)-\Psi(u_1-u_0)$. We also define
\begin{multline}
    \Jmp{\Delta_\sfc}(u,[a,b]):=\Delta_\sfc
    (a,u(a),\ur(a)) + \Delta_\sfc
    (b,\ul(b),u(b)) 
    \\ 
    +\sum_{t \in \Ju\cap (a,b)}
    \Delta_\sfc
    (t,\ul(t),u(t),\ur(t)),
\end{multline}
and the corresponding augmented total variation $\mVar{\Psi,\sfc}$ is then
\begin{equation} \label{eq:varP}
    \mVar{\Psi,\sfc}(u,[a,b]):=\varpsi(u,[a,b])+
    \Jmp{\Delta_\sfc}(u,[a,b]).
  \end{equation}
\end{definition}
The infimum in \eqref{eq:dissipationcost} is attained whenever there is at least one admissible transition $\vartheta$ with finite cost. In this case, we say that $\vartheta$ is an \emph{optimal transition}.
\begin{definition}[Optimal transitions] Let $t\in[a,b]$ and $u_-$,
  $u_+\in X$. 
We say that a curve $\vartheta\in \rmC(E;X)$, $E$ being a
compact subset of $\R$, is an optimal transition between $u_-$ and $u_+$ if
\begin{equation}
  u_-=\vartheta(E^- ),\quad
  u_+=\vartheta(E^+ ),\quad
  \Fd(t,u_-,u_+)=\Cf(t,\vartheta,E).
\end{equation}
$\vartheta$ is \emph{tight} if for every $I\in \Holes(E)$ 
$\vartheta(I^-)\neq \vartheta(I^+)$.
$\vartheta$ 
is a
\begin{align}
  \text{pure jump transition, if }&E\setminus \{E^-,E^+\}\text{ is discrete,}\\
  \text{sliding transition, if } &\Res(t,\vartheta(r))=0\quad \text{for every $r\in E$},  \\
  \text{viscous transition, if } &\Res(t,\vartheta(r))>0\quad \text{for every $r\in E\setminus \{E^{\pm}\}$} \label{eq:viscoustransition}.
\end{align}
\end{definition}
Notice that if $\vartheta$ is a transition with finite cost $\Cf(t,\vartheta,E)<\infty$, then the set
\begin{equation}
E_{\Res}:=\{r\in E \setminus E^+: \Res(t,	\vartheta(r)>0\}\quad\text{is discrete, i.e. all its points are isoltaed}.
\end{equation}
\par
With these notions at our disposal, we can now give the precise definition of Visco-Energetic solutions to the rate-independent system $(X,\cE,\Psi,\delta)$.
\begin{definition}[Visco-Energetic (VE) solutions] We say that a curve $u\in \BV ([a,b];X)$ 
  is a 
\emph{Visco-Energetic (VE) solution} of the rate-independent system $(X,\mathcal{E},\Psi,\delta)$ if it satisfies the stability condition
\begin{equation}\label{eq:stability}
  u(t)\in \SSD(t)\quad\text{for every }t\in [a,b]\setminus \Ju,
  \tag{S$_\sfD$}
\end{equation}
and the energetic balance
\begin{equation} \label{energybalance}
\mathcal{E}(t,u(t))+\mVar{\Psi,\sfc}(u,[a,t])=\mathcal{E}(a,u(a))+\int_a^t\cP(s,u(s))\,\rmd s \tag{$\mathrm{E_{\Psi,\sfc}}$}
\end{equation}
for every $t\in[a,b]$.
\end{definition}
\par

Existence of Visco-Energetic solutions in a much more general metric-topological setting  is proved in \cite{MinSav16}. Solutions are obtained as 
a limit of piecewise constant interpolant of 
discrete solutions $U^n_\tau$ obtained by recursively solving the modified time
\emph{Incremental Minimization Scheme}
\begin{equation}
  \label{eq:167}
  \min_{U\in X} \cE(t^n_\tau,U)+\sfD(U^{n-1}_\tau,U).
  \tag{IM$_\sfD$}
\end{equation} 
starting from an initial datum $U_\tau^0\approx u_0$.

\subsection{Some useful properties of VE solutions} \label{subsec:VE-Euclidean-properties}
In this section we collect a list of useful properties of Visco-Energetic solutions and we prove an equivalent characterization in the finite-dimensional setting, involving a doubly nonlinear evolution equation. For more details about these results and their proof we refer to \cite{MinSav16,Minotti16T}.
\par
To simply the notations, we first introduce the \emph{Minimal set}, which is related to the connection of two points through a step of Minimizing Movements.
\begin{definition}[Moreau-Yosida regularization and Minimal set] Suppose that $\cE$ satisfies \eqref{eq:E-assumption} and $\eqref{eq:W'-assumption}$. The $\sfD$-Moreau-Yosida regularization $\cY:[a,b]\times \R\to \R$ of $\cE$
  is defined by
  \begin{equation}
    \label{eq:111}
    \cY(t,u):=\min_{v\in\R}\cE(t,v)+\sfD(u,v).
  \end{equation}
  For every $t\in [a,b]$ and $u\in \R$ the minimal set is 
  \begin{equation}
    \label{eq:107}
    \rmM(t,u):=\argmin_{\R} \cE(t,\cdot)+\sfD(u,\cdot)=
    \Big\{v\in \R:\cE(t,v)+\sfD(u,v)=\cY(t,u)\Big\}.
  \end{equation}
\end{definition}
Notice that, by \eqref{eq:E-assumption} and \eqref{eq:W'-assumption},
$\rmM(t,u)\neq\emptyset$ for every $t,u$.
It is also clear that  $\cR(t,u)=\cE(t,u)-\cY(t,u)$ and that
\[
u\in \SSD(t)\Longrightarrow u\in \rmM(t,u).
\]
\par
As we have mentioned in the Introduction, when $t\in\Ju$ and $\vartheta:E\rightarrow\R$ is an optimal transition between $\ul(t)$ and $\ur(t)$, $\vartheta$ ``keeps trace'' of the whole construction via \eqref{eq:167}. For instance, when $\vartheta(E)$ is discrete, every point is obtained with a step of Minimizing Movements from the previous one, with the energy frozen the time $t$. The next result, \cite[Theorem 3.16]{MinSav16}, formalises this property and characterizes Visco-Energetic optimal transitions.  Whenever a set $E\subset\R$ is given, we will use the notations
\begin{equation} \label{eq:44}
r_E^-:=\sup\{E\cap (-\infty,r)\}\cup \{E^-\},\qquad r_E^+:=\inf\{E\cap(r,+\infty)\}\cup\{E^+\}.
\end{equation}
\begin{theorem} \label{prop:2}
 A curve $\vartheta\in \rmC(E,\R)$ with $\vartheta(E)\ni
  u(t)$ is an optimal transition between $\ul(t)$ and $\ur(t)$ satisfying
\begin{equation} \label{eq:180}
\cE(t,\ul(t))-\cE(t,\ur(t))=\Cf(t,\vartheta,E)
\end{equation}
  if and only if it satisfies 
  \begin{equation}
    \label{eq:120}
    \mVar\Psi(\vartheta,E\cap[r_0,r_1])\le
    \cE(t,\vartheta(r_0))-\cE(t,\vartheta(r_1))\quad
    \text{for every }r_0,r_1\in E,\ r_0\le r_1,
  \end{equation}
  and
  \begin{equation}
    \label{eq:121}
    \vartheta(r)\in \rmM(t,\vartheta(r^-_E))\quad \text{for every
    }r\in E\setminus \{E^-\}.
  \end{equation}
\end{theorem}
In some situations, the first inequality \eqref{eq:180} can be proved thanks to the following elementary lemma, whose proof is analogous to \cite[Lemma 6.1]{MinSav16}
\begin{lemma}
  \label{le:elementary}
  Let $E\subset \R$ be a compact set with $E^- <E^+ $, let $L(E)$ be the
  set of 
  limit
  points of $E$. 
  We consider a function $f:E\to \R$ lower semicontinuous 
  and continuous on the left and a function $g\in \rmC(E)$ 
  strictly increasing, satisfying the 
  following two conditions:\\
  i) for every $I\in \Holes(E)$ 
  \begin{equation}
    \label{eq:65}
    \frac{f(I^+ )-f(I^- )}
    {g(I^+ )-g(I^- )}\ge 1;
  \end{equation}
  ii) for every $t\in L(E)$ which is an accumulation point of
  $L(E)\cap (-\infty,t)$ we have
  \begin{equation}
    \label{eq:22}
    \limsup_{s\up t,\ s\in L(E)} \frac{f(t)-f(s)}
    {g(t)-g(s)}\ge 1.
  \end{equation}
  Then the map $s\mapsto f(s)-g(s)$ is non decreasing in $E$; in particular
  \begin{equation}
    \label{eq:52}
    f(E^+ )-f(E^- )\ge g(E^+ )-g(E^- ).
  \end{equation}
\end{lemma}

\par

The following proposition, a consequence of \eqref{eq:crinqualitytheta}, is useful to prove existence of VE solutions since it gives some sufficient conditions. 
\begin{proposition}[Sufficient criteria for VE solutions]
  \label{prop:leqinequality} Let 
  $u\in \BV([a,b];X)$ be a curve satisfying the
  stability condition \eqref{stability}.
  Then $u$ is a \VE\ solution of the
rate-independent system 
$(X,\mathcal{E},\Psi,\delta)$ if and only if it satisfies 
one of the following equivalent characterizations:
\begin{enumerate}[i)]
\item 
$u$ satisfies the $(\Psi,\sfc)$-energy-dissipation inequality
\begin{equation} \label{leqinequality}
    \mathcal{E}(b,u(b))+\mVar{\Psi,\sfc}(u,[a,b])\leq\mathcal{E}(a,u(a))+\int_a^b\cP(s,u(s))\rmd
    s. 
  \end{equation}
\item $u$ satisfies the $\sfd$-energy-dissipation inequality
  \begin{equation}
    \cE(t,u(t))+\varpsi(u,[s,t])\leq \cE(s,u(s))+\int_s^t\cP(r,u(r))\rmd
    r\quad\text{for all $s\le t\in[a,b]$}\label{eq:110}
\end{equation}
and the following jump conditions at each point $t\in \Ju$
\begin{align}\label{Jve}
  \tag{$\mathrm{J_{VE}}$}
\begin{split}
\mathcal{E}(t,u(t\leftl ))-\mathcal{E}(t,u(t))=\Fd(t,u(t\leftl ),u(t)), \\
\mathcal{E}(t,u(t))-\mathcal{E}(t,u(t\rightl ))=\Fd(t,u(t),u(t\rightl )), \\
\mathcal{E}(t,u(t\leftl ))- \mathcal{E}(t,u(t\rightl ))=\Fd(t,u(t\leftl ),u(t\rightl )).
\end{split} \end{align}
\end{enumerate}
\end{proposition}
\par
Another simple property concerns the behaviour of Visco-Energetic solutions with respect to restrictions and concatenation. The proof is trivial.
\begin{proposition}[Restriction and concatenation principle] \label{lem:1} The following properties hold:
\begin{enumerate}
\item The restriction of a Visco-Energetic solution in $[a,b]$ to an interval $[\alpha,\beta]\subseteq [a,b]$ is a Visco-Energetic solution in $[\alpha,\beta]$;
\item If $a=t_0<t_1<t_{m-1}<t_m=b$ is a subdivision of $[a,b]$ and $u:[a,b]\rightarrow\R$ is Visco-Energetic solution on each one of the intervals $[t_{j-1},t_j]$, then $u$ is a Visco-Energetic solution in $[a,b]$.
\end{enumerate}
\end{proposition}
\par
In our finite-dimensional setting it is possible to give another sufficient criterium for Visco-Energetic solutions, more precisely a characterization through the stability condition \eqref{stability}, a doubly nonlinear differential inclusion, and the Jump condition \eqref{Jve}. This result will be the starting point for our discussion in the one-dimensional case.
\begin{theorem}[Characterization of VE solutions]
  \label{prop:differential-characterization} 
  A curve $u\in \BV([a,b];X)$ is a Visco-Energetic solution of the rate-independent system $(X,\mathcal{E},\Psi,\delta)$ if and only if it satisfies the stability condition \eqref{stability}, the doubly nonlinear differential inclusion
  \begin{equation} \label{eq:DN0}
  \partial\Psi\left(\frac{\rmd u'_{\mathrm{co}}}{\rmd\mu}(t)\right)+\rmD W(u(t))\ni\ell(t)\quad\text{for $\mu$-a.e. $t\in(a,b)$,\quad $\mu:=\mathscr{L}^1+|u'_{\mathrm{co}}|$} \tag{$\mathrm{DN}_0$}
  \end{equation}
  and the jump conditions \eqref{Jve} at every $t\in\Ju$:
\begin{align}
  \tag{$\mathrm{J_{VE}}$}
\begin{split}
\mathcal{E}(t,u(t\leftl ))-\mathcal{E}(t,u(t))=\Fd(t,u(t\leftl ),u(t)), \\
\mathcal{E}(t,u(t))-\mathcal{E}(t,u(t\rightl ))=\Fd(t,u(t),u(t\rightl )), \\
\mathcal{E}(t,u(t\leftl ))- \mathcal{E}(t,u(t\rightl ))=\Fd(t,u(t\leftl ),u(t\rightl )).
\end{split} \end{align}
\end{theorem}

\begin{proof}
From the definition of the viscous dissipation cost $\cost(t,\ul(t),\ur(t))$, \ref{def:transition-cost}, it is immediate to check that
\[
\varpsi\big(u,[a,b]\big)\le \mVar{\Psi,\sfc}\big(u,[a,b]\big),
\]
so that Visco-Energetic solutions are in particular \emph{local solutions}, in the sense of \cite{Mielke-Rossi-Savare12}. This differential characterization is therefore an immediate consequence of Proposition \ref{prop:leqinequality} and \cite[Proposition 2.7]{Mielke-Rossi-Savare12}.
\end{proof}

\subsection{The one-dimensional setting} \label{subsec:one-dimensional-setting}
From now on we consider the particular case $X=\R$, which we also identify with $X^*$. We will denote by $v^+$, $v^-$ the positive and the negative part of $v\in\R$.
\paragraph{Dissipation.} A dissipation potential is a function of the form
\begin{equation}
\Psi(v):=\alpha_+v^++\alpha_-v^-,\quad v\in\R,\quad\text{for some $\alpha_\pm>0$.}
\end{equation}
Hence, we have
\[
\partial\Psi(v)=\begin{cases} \alpha_+ \quad &\text{if $v>0$,} \\ [-\alpha_-,\alpha_+] \quad&\text{if $v=0$,} \\ -\alpha_- \quad&\text{if $v<0$} \end{cases}\qquad\text{for all $v\in\R$},
\]
and
\begin{equation} \label{eq:40}
K^*=[-\alpha_-,\alpha_+],\quad\Psi_*(w)=\frac{1}{\alpha_+}w^+ +\frac{1}{\alpha_-}w^-\quad\text{for all $w\in\R$}.
\end{equation}

\paragraph{Energy functional.} The energy is given by a function $\cE:[a,b]\times\R\rightarrow\R$ of the form
\begin{equation} \label{eq:E-assumption}
\cE(t,u):=W(u)-\ell(t)u
\end{equation}
with $\ell\in \rmC^1([a,b])$ and $W:\R\rightarrow\R$ such that
\begin{equation} \label{eq:W'-assumption}
W\in \rmC^1(\R),\quad\lim_{x\rightarrow-\infty}W'(x)=-\infty,\quad\lim_{x\rightarrow+\infty}W'(x)=+\infty.
\end{equation}

\paragraph{Viscous correction.} The admissible one-dimensional viscous correction is a continuous map $\delta:\R\times\R\rightarrow [0,+\infty)$ which satisfies
\begin{equation}\label{eq:delta} \tag{$\delta1$}
\lim_{v\rightarrow u}\frac{\delta(u,v)}{|v-u|}=0\qquad\text{for every $u\in \SSD(t)$},\quad t\in [a,b],
\end{equation}
and the reverse triangle inequality 
\begin{equation} \label{eq:delta2} \tag{$\delta2$} 
\delta(u_0,u_1)> \delta(u_0,v)+\delta(v,u_1)\qquad \text{for every $u_0< v< u_1$}.
\end{equation}
We still use the notation $\sfD(u,v):=\Psi(v-u)+\delta(u,v)$ for the augmented dissipation.
\par

\begin{remark}[Admissible viscous corrections]\upshape
Assumption \eqref{eq:delta} is necessary for the general theory of Visco-Energetic solutions; \eqref{eq:delta2} will be crucial for our one-dimensional characterization (see section \ref{subsec:main-theorem}). However, these assumptions are quite natural: they are satisfied, for example, if we choose $\delta$ of the form
\[
\delta(u,v)=f(\Psi(u-v))\qquad\text{with $f$ positive, strictly convex, with $\lim_{r\rightarrow 0}\frac{f(r)}{r}=0$}.
\]
For instance, the standard choice $\delta(u,v)=\frac{\mu}{2}(v-u)^2$, for some positive parameter $\mu$, is admissible. This particular case will be analysed with some example in sections \ref{sec:3} and \ref{sec:4}.
\end{remark}

\section{Visco-Energetic solutions of rate-independent systems in \texorpdfstring{$\R$}{R}} \label{sec:3}
As we have underlined in the Introduction, Visco-Energetic solutions of the rate-independent system $(\R,\cE,\Psi,\delta$) are intermediate between energetic, which correspond to the choice $\delta\equiv 0$, and Balanced Viscosity solutions, which corresponds to a choice of $\delta=\delta(\tau)$, depending of $\tau$, in \eqref{eq:167}  of the form
\[
\delta_\tau(u,v):=\mu(|\tau|)\delta(u,v),\qquad \mu:(0,+\infty)\rightarrow(0,+\infty), \quad\lim_{r\rightarrow 0}\mu(r)=+\infty.
\]
Guided by the characterizations of this two cases, given in \cite{Rossi-Savare13} in a similar one-dimensional setting and recalled in the Introduction, we obtain a full characterization for the visco-energetic case. In particular, the main results of \cite{Rossi-Savare13} can be recover for some choices of $\delta$.

\subsection{One-sided global slopes with a \texorpdfstring{$\delta$}{d} correction} \label{subsec:stability}
One-sided global slopes are used in \cite{Rossi-Savare13} to give a one-dimensional characterization of Energetic solutions of the rate-independent system $(\R,\cE,\Psi)$. We recall their definitions:
\begin{equation}\label{eq:one-sided-slopes}
\Wir(u):= \inf_{z>u} \frac{W(z)-W(u)}{z-u},\qquad \Wsl(u):= \sup_{z<u} \frac{W(z)-W(u)}{z-u},
\end{equation}
where the subscripts $\mathsf{ir}$ and $\mathsf{sl}$ stands for \emph{inf-right} and \emph{sup-left} respectively.
\par
In this section we introduce a generalization of $W'_{\sfi\sfr}$ and $W'_{\sfs\sfl}$, and we prove some important properties. These slopes allow us to give an equivalent, one-dimensional, characterization of the $\sfD$-Stability \eqref{eq:stability}. 
\begin{definition} For every $u\in \R$ we define the one-sided global slopes with a $\delta$ correction 
\begin{gather}
\Wird\delta(u):=\inf_{z>u}\left\{\frac{1}{z-u}\Big(W(z)-W(u)+\delta(u,z)\Big)\right\}, \label{eq:Wir} \\
\Wsld\delta(u):=\sup_{z<u}\left\{\frac{1}{z-u}\Big(W(z)-W(u)+\delta(u,z)\Big)\right\}.
\label{eq:Wsl}
\end{gather}
\end{definition}
For simplicity, we will still use the notations $\Wir$ and $\Wsl$ instead of $\Wird0$ and $\Wsld0$ when $\delta\equiv 0$. From \eqref{eq:delta} it follows that the modified global slopes satisfy 
\begin{equation} \label{eq:4}
\Wird\delta(u)\le W'(u)\le \Wsld\delta(u), \quad \text{for every $u \in \R$}
\end{equation}
and it is not difficult to check they are continuous. Indeed, it is sufficient to introduce the continuous function $V: \R\times \R\rightarrow \R$
\[
V(u,z):=\begin{cases} W'(u) &\quad \text{if $z=u$}, \\ \frac{1}{z-u}\Big(W(z)-W(u)+\delta(u,z)\Big)&\quad \text {if $z\neq u$}, \end{cases}
\]
and observe, e.g. for $\Wird\delta$, that 
\[
\Wird\delta (u) = \min \{V(u,z): z\ge u \}
\]
and for $u$ in a bounded set the minimum is attained in a compact set thanks to \eqref{eq:W'-assumption}.
\par
If $\delta$ is big enough, in a suitable sense, equalities hold in \eqref{eq:4}. An important result is stated in the following proposition. 
\begin{proposition} \label{prop:1}
Suppose that $W$ satisfies the  $\delta$-convexity assumption
\begin{equation} \label{eq:41}
W(v)\le (1-t)W(u)+tW(w)+t(1-t)\delta(u,w),\qquad v=(1-t)u+tw,\quad t\in[0,1].
\end{equation}
Then the one-sided slopes coincides with the usual derivative:
\[
\Wird{\delta}(u)=W'(u)=\Wsld{\delta}(u)\qquad \text{for every $u\in\R$}.
\]
\end{proposition}

\begin{proof}
We prove the first equality since the second one is analogous. Let us take $v,w\in\R$ with $u<v\le w$ and $t\in (0,1]$ such that $v=(1-t)u+tw$. Then
\begin{multline*}
\frac{W(v)-W(u)}{v-u}\le \frac{(1-t)W(u)+tW(w)+ t(1-t)\delta(u,w)-W(u)}{t(w-u)} = \\
\frac{W(w)-W(u)+\delta(u,w)}{w-u} -t\frac{\delta(u,w)}{w-u}\le  \frac{W(w)-W(u)+\delta(u,w)}{w-u}.
\end{multline*}
Passing to the limit as $v\downarrow u$ we get
\[
W'(u)\le \frac{W(w)-W(u)+\delta(u,w)}{w-u}\quad \text{for every $w>u$.}
\]
Now it is enough to take to infimum over $w>u$.
\end{proof}

\begin{remark} \upshape An interesting consequence of Proposition \ref{prop:1} is that if $W$ satisfies the usual $\lambda$-convexity assumption
\begin{equation} \label{eq:lambda-convex}
W(v)\le (1-t)W(u)+tW(w)-\lambda t(1-t)(w-u)^2,\qquad v=(1-t)u+tw,\quad t\in[0,1]
\end{equation}
for some $\lambda\in \R$, then for every $\mu\ge \min\{-\lambda, 0\}$ we can choose $\delta(u,w):=\mu(w-u)^2$ and \eqref{eq:41} holds. In particular, if $W$ is convex, for every admissible viscous correction $\delta$ the one-sided global slopes coincide with the usual derivative.
\end{remark}

\par
If $\Wird\delta (u)<W'(u)$ in a point $u\in\R$, then from \eqref{eq:W'-assumption} there exist $z>u$ which attains the infimum in \eqref{eq:Wir}. The same happens if $\Wsld\delta(u)>W'(u)$. Moreover, from the continuity of $W$ and of the global slopes, there exist a neighborhood of $u$ in which the strict inequality holds. In this neighborhood $\Wird\delta$, or $\Wsld\delta$, are decreasing.

\begin{proposition} \label{prop: Wir.prop} Let $I \subseteq \R$ be an open interval such that 
\[
\Wird\delta (v) < W'(v) \quad\text{(resp. $\Wsld\delta(v)>W'(v)$)} \qquad \text{ for every $v\in I$}. 
\]
Then $\Wird\delta$ (resp. $\Wsld\delta$) is decresing on $I$.
\end{proposition}

\begin{proof}  Let $v_1\in I$ and let $z>v_1$ be an element that attains the infimum in \eqref{eq:Wir}. Then for every $v_2<z$ we have the inequality
\[
\Wird\delta(v_2)-\Wird\delta(v_1) \le \\ \frac{W(z)-W(v_2)+\delta(v_2,z)}{z-v_2}- \left(\frac{W(z)-W(v_1)+\delta(v_1,z)}{z-v_1}\right).
\]
From \eqref{eq:delta2}, $\delta(v_1,z)\ge\delta(v_2,z)$ so that
\[
\frac{1}{v_2-v_1}\left[\frac{\delta(v_2,z)}{z-v_2}-\frac{\delta(v_1,z)}{z(v_1)-v_1}\right]\le \frac{\delta(v_2,z)}{(z-v_2)(z-v_1)}.
\]
Combining this with the simple identity
\[
\frac{W(z)-W(v_1)}{z-v_1}= \\
\frac{W(z)-W(v_2)}{z-v_2}\left(1-\frac{v_2-v_1}{z-v_1}\right)+\frac{W(v_2)-W(v_1)}{v_2-v_1}\frac{v_2-v_1}{z-v_1}
\]
after a simple computation we obtain
\[
\frac{\Wir (v_2)-\Wir (v_1)}{v_2-v_1} \le 
\frac{1}{z-v_1}\left[\frac{W(z)-W(v_2)+\delta(v_2,z)}{z-v_2}- 
\frac{W(v_2)-W(v_1)}{v_2-v_1}\right].
\]
Passing to the limsup for $v_2\downarrow v_1$ we get
\[
\limsup_{v_2\downarrow v_1} \frac{\Wird\delta (v_2)-\Wird\delta(v_1)}{v_2-v_1}\le \frac{1}{z-v_1}\left(\Wird\delta (v_1) - W'(v_1)\right) < 0.
\]
The claim follows from a classical result concerning Dini derivatives, see \cite{Gal57}.
\end{proof}

\paragraph{Characterizations of $\sfD$-Stability.} 
Taking \eqref{eq:Wir} and \eqref{eq:Wsl} into account, we can formulate a characterization of the global $\sfD$-stability \eqref{stability}. Since the energy is of the form $\cE(t,u)=W(u)-\ell(t)u$, \eqref{stability} is equivalent to
\[
W(u(t))-W(v)-\ell(t)(u(t)-v)\le \Psi(v-u)+\delta(u,v)\quad\text{for every $t\in[a,b]\setminus \Ju$, $v\in \R$}.
\]
Dividing by $u(t)-v$ and taking the infimum over $v>u(t)$, or the supremum over $v<u(t)$, for every $t\in[a,b]\setminus \Ju$ we get the system of inequalities 
\begin{equation} \label{eq:1d-stability}
-\alpha_- \le \ell(t)-\Wsld\delta(u(t))\le \ell(t)-W'(u(t))\le \ell(t)-\Wird\delta(u(t))\le\alpha_+, \tag{$\rmS_{\sfD, \R}$}
\end{equation}
which are the one-dimensional version of the global $\sfD$-stability. The continuity property of the $\delta$-corrected one-sided slopes also yields for every $t\in (a,b)$
\begin{gather} 
-\alpha_- \le \ell(t)-\Wsld\delta(\ur(t))\le \ell(t)-W'(\ur(t))\le \ell(t)-\Wird\delta(\ur(t))<\alpha_+, \label{eq:1d-stabilityright}  \\
-\alpha_- \le \ell(t)-\Wsld\delta(\ul(t))\le \ell(t)-W'(\ul(t))\le \ell(t)-\Wird\delta(\ul(t))<\alpha_+.   \label{eq:1d-stabilityleft}
\end{gather}

\begin{remark}\upshape The stability region $\SSD$ is bigger when $\delta$ increases. If we call
\begin{equation}
\SS_\infty:=\{(t,u)\in [a,b]\times \R: \rmD_u\cE(t,u)\in K^*\},
\end{equation}
where $K^*$ is defined in \eqref{eq:K*}, the set of points which satisfies the \emph{local stability} condition typical of BV solutions, \cite{Mielke-Rossi-Savare12,Mielke-Rossi-Savare13}, it is immediate to check that 
\[
\SSd\subseteq\SSD \subseteq \SS_\infty\qquad\text{for every admissible viscous correction}.
\]
The first inclusion is an equality if $\delta\equiv 0$. If the energy satisfies the $\delta$-convexity property \eqref{eq:41}, or, equivalently, if $\delta$ is chosen big enough, from Propostion \ref{prop:1} we get $\SSD=\SS_\infty$. 

\end{remark}

\subsection{Visco-Energetic Maxwell rule}
After the brief discussion about stability in section \ref{subsec:stability}, we now focus on jumps. In this section we show a relation between the minimal sets \eqref{eq:107} and the one-sided global slopes $\Wird\delta$ and $\Wsld\delta$, along with some geometrical interpretations of the results.
\begin{proposition} \label{prop:3}
Let $t,u\in\R$. Suppose that $z\in \rmM(t,u)$. Then 
\begin{gather}
\Wird\delta(v)\le\frac{W(z)-W(v)}{z-v}+\frac{\delta(v,z)}{z-v}< \ell(t)-\alpha_+ \qquad \text{if $u< v<z$}, \label{eq:8} \\
\Wsld\delta(v)\ge \frac{W(z)-W(v)}{z-v}+\frac{\delta(v,z)}{z-v}> \ell(t)+\alpha_- \qquad \text{if $u> v>z$}, \label{eq:9}
\end{gather}
Moreover, if $u\in \SSD(t)$ the following identities hold:
\begin{gather}
\Wird\delta(u)=\frac{W(z)-W(u)}{z-u}+\frac{\delta(u,z)}{z-u}=\ell(t)-\alpha_+ \qquad \text{if $z>u$},\label{eq:10} \\
\Wsld\delta(u)=\frac{W(z)-W(u)}{z-u}+\frac{\delta(u,z)}{z-u}= \ell(t)+\alpha_- \qquad \text{if $z<u$}. \label{eq:11}
\end{gather}
 \end{proposition}
 \begin{proof}
Let us consider the case $z>u$. From the minimality of $z$ for every $v\in (u,z)$ we get
\[
W(z)-W(v)-\ell(t)(z-v)\le -\alpha_+(z-v)+\delta(u,v)-\delta(u,z).
\]
Taking $\eqref{eq:delta2}$ into account and dividing by $z-u$ we get
\[
\frac{W(z)-W(v)}{z-v}-\ell(t)< -\alpha_+ -\frac{\delta(v,z)}{z-v},
\]
which proves \eqref{eq:8}. If $u\in \SSD(t)$, we can combine the one dimensional $\sfD$-stability condition \eqref{eq:1d-stability} with \eqref{eq:8}, where we pass to the limit for $v\downarrow u$, and we get 
\[
\Wird\delta(u)\le\frac{W(z)-W(u)}{z-u}+\frac{\delta(u,z)}{z-u}\le\ell(t)-\alpha_+\le \Wird\delta(u),
\]
so that all the previous inequalities are identities and \eqref{eq:10} is proved. The case $z<u$ can be proved in a similar way.
\end{proof}
\begin{remark} \label{rem:1} \upshape Notice that the strict inequality in \eqref{eq:delta2} implies
\begin{align}
\text{if $z\in \rmM(t,u)$, $z>u$, then}\qquad \Wird\delta(v)<\ell(t)-\alpha_+\quad \forall v\in (u,z), \label{eq:16}\\
\text{if $z\in \rmM(t,u)$, $z<u$, then}\qquad \Wsld\delta(v)>\ell(t)-\alpha_-\quad \forall v\in (z,u). \label{eq:17}
\end{align}
In particular $v\not\in \SSD(t)$, since \eqref{eq:16} and \eqref{eq:17} contradict the global stability \eqref{eq:1d-stability}. This inequalities will be one the key ingredients for the characterization Theorem \ref{thm:main-theorem}.
\end{remark}

\paragraph{$\sfD$-Maxwell rule.} Equalities \eqref{eq:10} and \eqref{eq:11} admit a nice geometrical interpretation. Suppose that $u$ is a Visco-Energetic solution, $t\in\Ju$ and that there exist $z\in \rmM(t,\ul(t))$ with $z>\ul(t)$. According to \eqref{eq:1d-stabilityleft}, $\ul(t)$ is stable, so that we can choose $u=\ul(t)$ in \eqref{eq:10} and we get 
\begin{equation} \label{eq:12}
W(z)=W(\ul(t))+(\ell(t)-\alpha_+)(z-\ul(t))-\delta(\ul(t),z).
\end{equation}
This identity is a generalization of the so-called \emph{Mawell rule}: in the energetic case, combining global stability and energetic balance, we easily get $z=\ur(t)$, so that \eqref{eq:12} assume the classical formulation
\begin{equation}
\int_{\ul(t)}^{\ur(t)}\Big( W'(r)-\ell(t)+\alpha_+ \Big)\,\rmd r=0.
\end{equation} 
\par
Considering for simplicity the choice $\delta(u,v):=\frac{\mu}{2}(v-u)^2$, for some parameter $\mu>0$, when $W'(\ul(t))=\ell(t)-\alpha_+$ \eqref{eq:12} can be rewritten in the form
\[
W(z)=W(\ul(t))+W'(\ul(t))(z-\ul(t))-\frac{\mu}{2}(z-\ul(t))^2.
\]
This means that we can have a jump only when the area between the graph $W'$ and the straight line whose slope is $-\mu$ vanishes. If $\mu$ is big enough, then the area is always positive and $\rmM(t,u)=\{u\}$. In this case the description of the jump transition will be more complicated (see section \ref{subsec:main-theorem} and \ref{sec:4} for more details).




\subsection{Main characterization Theorem} \label{subsec:main-theorem}
In this section we exhibit an explicit characterization of Visco-Energetic solutions for a general (i.e. non monotone) external loading $\ell$. This result is the equivalent of \cite[Theorem 3.1]{Rossi-Savare13} and \cite[Theorem 5.1]{Rossi-Savare13} for Energetic and BV solutions. 
\begin{theorem}[1d-characterization of VE solutions] \label{thm:main-theorem} Let $u\in \mathrm{BV}([a,b];\R)$ be a Visco-Energetic solution of the rate-independent system $(\R,\cE,\Psi,\delta)$. Then the the following properties hold:
\begin{enumerate}[a)]
\item $u$ satisfies the 1d-stability condition \eqref{eq:1d-stability} for every $t\in[a,b]\setminus \Ju$ (and therefore \eqref{eq:1d-stabilityright} and \eqref{eq:1d-stabilityleft} as well);
\item $u$ satisfies the following precise formulation of the doubly nonlinear differential inclusion:
\begin{gather}
W'(\ur(t))=\Wird\delta(\ur(t))=\ell(t)-\alpha_+ \quad \text{ for every $t\in \mathrm{supp}\left((u')^+\right)\cap [a,b)$}, \label{eq:equation1} \\
W'(\ur(t))=\Wsld\delta(\ur(t))=\ell(t)+\alpha_- \quad \text{ for every $t\in \mathrm{supp}\left((u')^-\right)\cap [a,b)$}; \label{eq:equation2}
\end{gather} 
\item at each point $t\in \Ju$, $u$ fulfils the jump conditions
\begin{equation} \label{eq:jump1}
\min(\ul(t),\ur(t))\le u(t) \le \max(\ul(t),\ur(t))
\end{equation}
and
\begin{equation} \label{eq:jump2}
\Wird\delta(v)\le \ell(t)-\alpha_+\quad \text{if $\ul(t)< \ur(t)$},\qquad \Wsld\delta(v)\ge \ell(t)+\alpha_-\quad \text{if $\ul(t)>\ur(t)$},
\end{equation}
for every $v$ such that $\min\left(\ul(t),\ur(t)\right)\le v\le \max\left(\ul(t),\ur(t)\right)$. 
\end{enumerate}
\par 
\noindent Conversely, let $u\in\BV([a,b],\R)$ be a curve satisfying \eqref{eq:equation1}, \eqref{eq:equation2}, \eqref{eq:jump1}, \eqref{eq:jump2}, along with the following modified version of a):
\begin{enumerate}[a)]
\item[a')] $u$ satisfies the 1d-stability condition \eqref{eq:1d-stability} for every $t\in (a,b)$.
\end{enumerate}
Then $u$ is a Visco-Energetic solution of the rate-independent system $(\R,\cE,\Psi,\delta)$.
\end{theorem}
\par
Since any jump point belongs either to the support of $\left(u'\right)^+$ or of $\left(u'\right)^-$, combining \eqref{eq:equation1}, \eqref{eq:jump1}, \eqref{eq:jump2} and \eqref{eq:equation2}, \eqref{eq:jump1} and \eqref{eq:jump2} we also get at every $t\in \Ju\cap (a,b)$
\begin{gather}
\Wird\delta(\ul)=\Wird\delta(\ur)=W'(\ur)=\ell(t)-\alpha_+\qquad \text{ if $\ul<\ur$}, \label{eq:6}\\
\Wsld\delta(\ul)=\Wsld\delta(\ur)=W'(\ur)=\ell(t)-\alpha_-\qquad \text{ if $\ul>\ur$}, \label{eq:7}
\end{gather}
and this identities still hold in $t=a$ or $t=b$ if $u(a)\in S_\sfD(a)$ or $u(b)\in S_\sfD(b)$.

\begin{remark} \upshape For a full characterization of Visco-Energetic solutions we need that \eqref{eq:1d-stability} holds also when $t\in \Ju$. This condition is required just to recover the first and the second equalities in \eqref{Jve} from the third. 
However, it is quite natural: if $u$ is a Visco-Energetic solution, we can consider the right continuous function 
\[
\tilde{u}\in \BV([a,b],\R)\quad\text{such that $\tilde{u}(t):=\ul(t)$ for every $t\in[a,b]$}.
\]
Then $\tilde{u}$ is still a Visco-Energetic solution and $\tilde{u}(t)$ is stable for every $t\in(a,b]$.
\end{remark}

\begin{proof}[Proof of theorem \ref{thm:main-theorem}] We split the argument in various steps.
\par
 \underline {Claim 1}. \emph{ $\sfD$-stability \eqref{stability} is equivalent to \eqref{eq:1d-stability}}.\newline It is a consequence of the choice $\cE(t,u)=W(u)-\ell(t)u$; see the discussion in section \ref{subsec:stability}.
\par
\underline{Claim 2}. \emph{\eqref{Jve} implies the jump conditions \eqref{eq:jump1} and \eqref{eq:jump2}}.\newline
From the general properties of the viscous dissipation cost, there exist an optimal transition $\vartheta\in \rmC(E;\R)$ connecting $\ul(t)$ and $\ur(t)$, namely
\[
\vartheta(E^-)=\ul(t),\quad\vartheta(E^+)=\ur(t),\quad \Fd(t,\ul(t),\ur(t))=\Cf(t,\vartheta,E).
\]
Since \eqref{Jve} holds, we can apply Theorem \ref{prop:2}. Let us start from the case $\ul(t)<\ur(t)$ and let $v\in [\ul(t),\ur(t)$]. If $v\notin \vartheta(E)$, which is compact, there exist an open interval $I\subset [\ul(t),\ur(t)]\setminus \vartheta(E)$ such that $v\in I$. From \eqref{eq:121}
\[
\vartheta(I^+)\in\rmM (t,\vartheta(I^-)),
\]
so that, by Proposition \ref{prop:3}, we get $\Wird\delta(v)\le \ell(t)-\alpha_+$. By continuity, the inequality still holds if $v\in \vartheta(E)$ is isolated in $\vartheta(E)\cap[v,+\infty)$. Otherwise, $v\in \rmL\big(\vartheta(E)\cap [v,+\infty)\big)$, where $\rmL$ denotes the set of the limit points. From \eqref{eq:120} we have
\[
\cE(t,v)\ge \cE(t,v_1)+\alpha_+(v_1-v)\quad \text{for every $v_1\in \vartheta(E),\quad v_1> v $},
\]
which yields
\[
\Wird\delta(v)\le \frac{W(v_1)-W(v)}{v_1-v}+\frac{\delta(v,v_1)}{v_1-v}\le \ell(t)-\alpha_+ +\frac{\delta(v,v_1)}{v_1-v}.
\]
We can pass to the limit for $v_1\downarrow z$ so that $\eqref{eq:jump2}$ holds in $[\ul(t),\ur(t))$. By continuity, it still holds in $v=\ur(t)$.  The case $\ul(t)>\ur(t)$ can be proved in a similar way.
\par
The property \eqref{eq:jump1} easily follows by summing the identities of the jump conditions \eqref{Jve}, thus obtaining
\[
\cost(t,\ul(t),\ur(t))=\cost(t,\ul(t),u(t))+\cost(t,u(t),\ur(t)),
\]
and considering the additivity of the cost \eqref{eq:195}.
\par

\underline{Claim 3}. \emph{The jump conditions \eqref{eq:jump1}, \eqref{eq:jump2} and a') imply \eqref{Jve}}.\newline
Let us start again with $\ul(t)<\ur(t)$. We still want to apply Theorem \ref{prop:2}: we need to find an admissible transition $\vartheta\in \rmC(E;\R)$ which satisfies \eqref{eq:120} and \eqref{eq:121}. To define such a transition, let us consider
\[
S:=\{v\in [\ul(t),\ur(t)]: \Wird\delta(v)=\ell(t)-\alpha_+\text{ and } \Wsld\delta(v)\le \ell(t)+\alpha_- \}.
\]
The set $S$ is compact, then there exists a sequence of disjoint open intervals $I_k$ such that $[S^-,S^+]\setminus S=\bigcup_{k=0}^\infty I_k$. 
Let us fix for a moment one of these $I_k$. Taking into account assumption \eqref{eq:W'-assumption}, we can have only two possibilities.
\par
\begin{enumerate}[a)]
\item[-] \emph{Case 1: ``The initial jump''}. The infimum in $\Wird\delta(I_k^-)$ is attained in a point $z>I_k^-$.\newline
From $\Wird\delta(I_k^-)=\ell(t)-\alpha_+$ and \eqref{eq:10} we recover the energetic balance
\begin{equation}
\cE(t,z)+\sfD(I_k^-,z)=\cE(t,I_k^-).
\end{equation}
Arguing as in Proposition \ref{prop:3}, $\Wird\delta(v)<\ell(t)-\alpha_+$ for every $v\in(I_k^-,z)$, so that $z\in \overline{I_k}$.
We can thus define by induction the sequence $(u_n^k)$ such that
\[
u^k_0:=z,\qquad u_{n+1}^k=u_n^k\quad\text{if $u_n^k=I_k^+$} \qquad u^k_{n+1}\in \rmM(t,u_n^k) \quad \text{otherwise}.
\]
Notice that from Proposition \ref{prop:3} and Remark \ref{rem:1}, by induction we easily get $u_n^k\in \overline{I_k}$ for every $n\in \N$. Moreover,
\begin{equation}
\Psi(u_{n+1}^k-u_n^k)\le \cE(t,u_n)-\cE(t,u_{n+1}),
\end{equation}
so that $(u_n^k)$ is a Cauchy sequence and then it converges to some $\bar{u}^k\in \overline{I_k}$. From the general properties of the residual stability function
\begin{equation} \label{eq:19}
\Res(t,u_n^k)=\cE(t,u_n^k)-\cE(t,u_{n+1}^k)-\sfD(u_n^k,u_{n+1}^k).
\end{equation}
By passing to the limit in \eqref{eq:19} we get
\[
\Res(t,\bar{u}^k)=0,\quad \text{so that $\bar{u}^k\in S$},
\]
which means $\bar{u}^k\in\{I_k^-;I_k^+\}$. In addition, $\bar{u}^k\neq I_k^-$ since $\cE(t,u_{n+1}^k)<\cE(t,u_n^k)$  every time that $u_{n+1}^k\neq u_n^k$, which implies $\cE(t,\bar{u}^k)<\cE(t,I_k^-)$.
Finally, we conclude $\bar{u}^k=I_k^+$ and we set $E_k:=\bigcup_{n=0}^\infty \{u_n^k\}$.

\item[-] \emph{Case 2: ``The (double) chain''}. $W'(I_k^-)< \frac{W(z)-W(I_k^-)+\delta(I_k^-,z)}{z-I_k^-}$ for every $z>I_k^-$.\newline
In this case $\Wird\delta(I_k^-)=W'(I_k^-)=\ell(t)-\alpha_+$. The energy $\cE(t,u)=W(u)-(W'(I_k^-)+\alpha_+)u$ has negative derivative in $u=I_k^-$, so that it is decreasing in a neighborhood of $I_k^-$. Let us choose $\varepsilon>0$ such that $\cE(t,I_k^-+\varepsilon)<\cE(t,I_k^-)$. We can thus define by induction the following sequence $(u^k_{n,\varepsilon})$: 
\[
u_{0,\varepsilon}^k:=I_k^-+\varepsilon,\qquad u_{n+1,\varepsilon}^k=u_{n,\varepsilon}^k\quad\text{if $u_{n,\varepsilon}^k=I_k^+$}, \qquad u_{n+1,\varepsilon}^k\in\rmM(t,u_{n,\varepsilon}^k)\quad\text{otherwise}.
\]
As in the previous case, this sequence is well defined and it converges to $I_k^+$. In order to pass to the limit for $\varepsilon\downarrow 0$, we apply a compactness argument: we consider the family of sets
\[
E_{k,\varepsilon}:=\bigcup_{n=0}^\infty\{u_{n,\varepsilon}^k\}\cup \{I_k^+\}.
\]
$E_{k,\varepsilon}$ are compact and $E_{k,\varepsilon}\subseteq \overline{I_k}$. We can apply Kuratowski Theorem (see e.g.~\cite{Kuratowski55}): there exists a compact subset $E_k\subseteq \overline{I_k}$ such that, up to a subsequence, $E_{k,\varepsilon}\rightarrow E_k$ in the Hausdorff metric.
It is easy to check, \cite[Lemma 3.11]{MinSav16}, that $E_k^-=I_k^-$, $E_k^+=I_k^+$ and
\begin{equation} \label{eq:28}
z\in \rmM(t,z_{E_k}^-)\qquad\text{for every $z\in E_k$},
\end{equation}
where $z_{E_k}^-$ is defined in \eqref{eq:44}.
\end{enumerate}
In conclusion, we repeat this construction for every open interval $I_k$ and we consider $E:=\bigcup_{k=0}^{\infty} E_k\cup S$. Notice that  $E^-=\ul(t)$, $E^+=\ur(t)$ and $E$ is a compact subset of $\R$. Indeed, $E$ is bounded and if $(x_n)$ is a sequence in $E$ that accumulates in some point $\bar{x}$, by construction $x_n$ is definitively contained in one of the sets $E_k$ or in $S$, which are compact. 
\par
We can thus consider the curve 
\[ 
\vartheta:E\rightarrow \R \quad\text{such that $\vartheta(z)=z$ for every $z\in E$}:
\]
it is an admissible transition connecting $\ul(t)$ and $\ur(t)$, with $\vartheta(E)\ni u(t)$ thanks to $a')$. It remains just to prove that $\vartheta$ satisfies \eqref{eq:120} and\eqref{eq:121}.
\par
Concerning \eqref{eq:120}, for every $I\subset \frH(E)$, by construction $\vartheta(I^+)\in\rmM(t,\vartheta(I^-)$, so that
\begin{equation} \label{eq:45}
\varpsi(\vartheta, E\cap [I^-,I^+])=\Psi(\vartheta(I^+)-\vartheta(I^-))\le \cE(t,\vartheta(I^-))-\cE(t,\vartheta(I^+)). 
\end{equation}
When $s\in \rmL\big(\vartheta(E)\cap (-\infty,s]\big)$ we get $\vartheta(s)\in S$, so that $\Wsld\delta(\vartheta(s))\le \ell(t)+\alpha_-$. In particular, since $\vartheta$ is increasing
\[
\frac{W(\vartheta(r))-W(\vartheta(s))+\delta(\vartheta(s),\vartheta(r))}{\vartheta(r)-\vartheta(s)}\le \ell(t)+\alpha_-\qquad\text{for every $r<s$}.
\] 
After a trivial computation, by using \eqref{eq:delta} and by passing to the limit we get
\begin{equation} \label{eq:46}
\limsup_{r\uparrow s}\frac{\cE(t,\vartheta(r))-\cE(t,\vartheta(s))}{\varpsi\big(\vartheta,E\cap [r,s]\big)}\ge 1.
\end{equation}
We can thus recover \eqref{eq:120} from \eqref{eq:45} and \eqref{eq:46} by using Lemma \ref{le:elementary}, where we set $f(s):=-\cE(t,\vartheta(s))$ and $g(s):=\varpsi(\vartheta,E\cap [E^-,s])$.
\par
Finally, \eqref{eq:121} holds by construction if $r$ is isolated in $E\cap (-\infty,r]$. Otherwise, $r_{E^-}=r$ and it is still satisfied. In conclusion, by Theorem \ref{prop:2} $\vartheta$ is an optimal transition satisfying the third of \eqref{Jve}. Considering the restriction of $\vartheta$ on $E\cap [\ul(t),u(t)]$ and $E\cap [u(t),\ur(t)]$ we also get the first two identities of \eqref{Jve}.

\underline{Claim 4}. \emph{$b)$ is equivalent to the doubly nonlinear equation \eqref{eq:DN0}.} \newline
We notice that \eqref{eq:DN0} yields
\begin{equation} \label{eq:47}
W'(u(t))= \ell(t)-\alpha_+\qquad \text{for $\left(u'_{\rm{co}}\right)^+$-a.e. $t\in (a,b)$},
\end{equation}
so that \eqref{eq:equation1} holds by continuity and by \eqref{eq:1d-stabilityright} in $\rm{supp} \left(u'\right)^+\setminus \Ju$. On the other hand, for every $t\in \Ju\cap \,\mathrm{supp} \left(u'\right)^+$ we have $\ul(t)<\ur(t)$. From \eqref{eq:1d-stabilityright} and \eqref{eq:jump2}, $\ell(t)-\alpha_+=\Wir(u_r(t))$
and  then combining Proposition \ref{prop: Wir.prop} and \eqref{eq:jump2} again we get  
\[
W'(\ur(t))=\Wir(\ur(t))=\ell(t)-\alpha_+,
\]
which proves \eqref{eq:equation1}. The identities in \eqref{eq:equation2} follow by the same argument. 
\par
The converse implication is trivial since $\mu$ is diffuse and therefore $\ul(t)=\ur(t)=u(t)$ for $\mu$-a.e. $t\in (a,b)$. Then \eqref{eq:DN0} follows combining \eqref{eq:equation1}, \eqref{eq:equation2} and \eqref{eq:1d-stability}.
\end{proof}

The previous general result has a simple consequence: a Visco-Energetic solution is locally constant in a neighborhood of a point where the stability condition \eqref{eq:1d-stability} holds with a strict inequality.
\begin{corollary} \label{cor:main-theorem-corollary} Let $u\in \rm{BV}([a,b];\R)$ be a Visco-Energetic solution of the rate-independent system $(\R,\cE,\Psi,\delta)$. Then $u$ is locally constant in the open set
\[
\cI:=\left\{t\in [a,b]: -\alpha_-<\ell(t)-\Wsld\delta(u(t))\le \ell(t)-\Wird\delta(u(t))<\alpha_+\right\}.
\]
\end{corollary}

\begin{proof}
By \eqref{eq:jump2} any $t\in \cI$ is a continuity point for $u$; the continuity properties of $\Wird\delta(\cdot)$ and $\Wsld\delta(\cdot)$ then show that a neighborhood of $t$ is also contained in $\cI$, so that $\cI$ is open and disjoint from $\Ju$. Relations \eqref{eq:equation1} and \eqref{eq:equation2} then yield that 
\[
u'=0 \qquad \text{in the sense of distributions in $\cI$}, 
\] 
so that $u$ is locally constant.
\end{proof}

\paragraph{Example.} We conclude this section with the classic example of the double-well potential energy $W(u)=\frac{1}{4}(u^2-1)^2$.
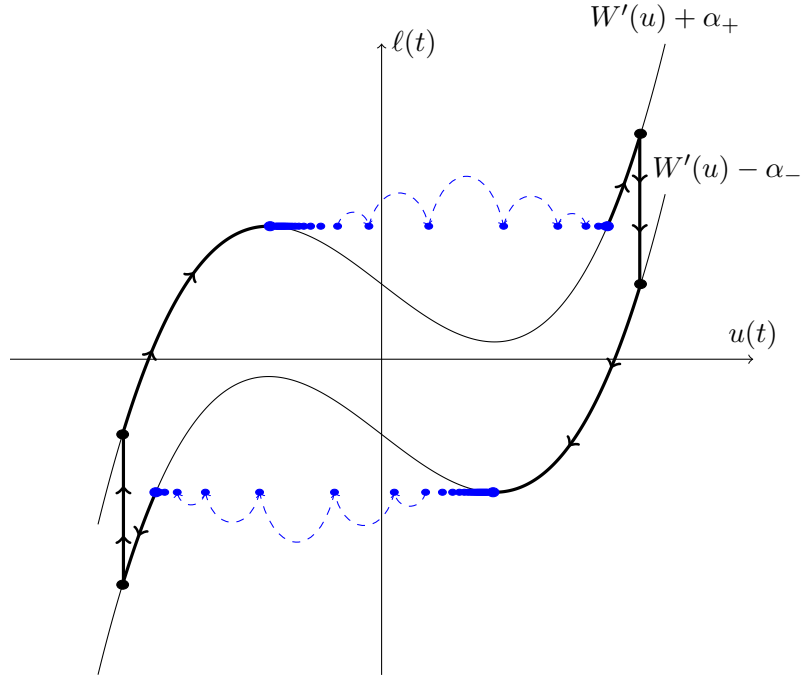
\begin{figure}[!ht]
\centering
\begin{tikzpicture}[xscale=1.3]

\draw[->] (-3.8,0) -- (3.8,0) node[above] {$u(t)$};
\draw[->] (0,-4.2) -- (0,4.2) node[right] {$\ell(t)$};

 \begin{scope}[scale=2]
    \draw[domain=-1.45:1.45,samples=200] plot ({\x}, {(\x)^3-\x+0.5}) node[above] {$W'(u)+\alpha_+$};
     \draw[domain=-1.324:-0.57,very thick,samples=200] plot ({\x}, {(\x)^3-\x+0.5}) [arrow inside={}{0.33,0.66}];
     \draw[domain=1.1547:1.32,very thick,samples=200] plot ({\x}, {(\x)^3-\x+0.5}) [arrow inside={}{0.50}]    ;
      \draw[domain=-1.45:1.45,samples=200] plot ({\x}, {(\x)^3-\x-0.5}) node[above] {\hspace{4em} $W'(u)-\alpha_-$};
     \draw[domain=-1.324:-0.57,very thick,samples=200] plot ({-\x}, {(-\x)^3+\x-0.5}) [arrow inside={}{0.33,0.66}];
     \draw[domain=1.1547:1.32,very thick,samples=200] plot ({-\x}, {(-\x)^3+\x-0.5}) [arrow inside={}{0.50}]    ;
     \foreach \Point in {(-0.57,0.8849),(1.1547,0.8849)}{
    \draw[fill=blue,blue] \Point circle(0.03);
}
     \foreach \Point in {(-0.542,0.8849),(-0.536,0.8849), (-0.530,0.8849), (-0.524,0.8849), (-0.518,0.8849), (-0.512,0.8849), (-0.506,0.8849), (-0.500,0.8849), (-0.494,0.8849), (-0.487,0.8849), (-0.479,0.8849), (-0.469,0.8849),(-0.457,0.8849),(-0.442,0.8849),(-0.423,0.8849),(-0.398,0.8849),(-0.363,0.8849),(-0.311,0.8849),(-0.225,0.8849),(-0.065,0.8849),(0.241,0.8849),(0.624,0.8849),(0.901,0.8849),(1.045,0.8849),(1.109,0.8849),(1.136,0.8849),(1.147,0.8849),(1.151,0.8849),(1.153,0.8849)}{
    \draw[fill=blue, blue] \Point circle(0.02);
}

     \foreach \Point in {(0.57,-0.8849),(-1.1547,-0.8849)}{
    \draw[fill=blue,blue] \Point circle(0.03);
}

     \foreach \Point in {(1.324,1.5),(1.324,0.5),(-1.324,-1.5),(-1.324,-0.5)}{
    \draw[fill=black] \Point circle(0.03);
}

  \foreach \Point in {(0.542,-0.8849),(0.536,-0.8849), (0.530,-0.8849), (0.524,-0.8849), (0.518,-0.8849), (0.512,-0.8849), (0.506,-0.8849), (0.500,-0.8849), (0.494,-0.8849), (0.487,-0.8849), (0.479,-0.8849), (0.469,-0.8849),(0.457,-0.8849),(0.442,-0.8849),(0.423,-0.8849),(0.398,-0.8849),(0.363,-0.8849),(0.311,-0.8849),(0.225,-0.8849),(0.065,-0.8849),(-0.241,-0.8849),(-0.624,-0.8849),(-0.901,-0.8849),(-1.045,-0.8849),(-1.109,-0.8849),(-1.136,-0.8849),(-1.147,-0.8849),(-1.151,-0.8849),(-1.153,-0.8849)}{
    \draw[fill=blue,blue] \Point circle(0.02);
}

\draw[very thick,arrow inside={}{0.33,0.66}] (1.32,1.5) -- (1.32,0.5);
\draw[very thick,arrow inside={}{0.33,0.66}] (-1.32,-1.5) -- (-1.32,-0.5);

\draw[->, dashed,blue,looseness=2] (-0.225,0.8849) to [bend left=80] (-0.065,0.8849);
\draw[->, dashed,blue,looseness=2.5] (-0.065,0.8849) to [bend left=80] (0.241,0.8849);
\draw[->, dashed,blue,looseness=3] (0.241,0.8849) to [bend left=80] (0.624,0.8849);
\draw[->, dashed,blue,looseness=2.5] (0.624,0.8849) to [bend left=80] (0.901,0.8849);
\draw[->, dashed,blue,looseness=2] (0.901,0.8849) to [bend left=80] (1.045,0.8849);

\draw[->, dashed,blue,looseness=2] (0.225,-0.8849) to [bend left=80] (0.065,-0.8849);
\draw[->, dashed,blue,looseness=2.5] (0.065,-0.8849) to [bend left=80] (-0.241,-0.8849);
\draw[->, dashed,blue,looseness=3] (-0.241,-0.8849) to [bend left=80] (-0.624,-0.8849);
\draw[->, dashed,blue,looseness=2.5] (-0.624,-0.8849) to [bend left=80] (-0.901,-0.8849);
\draw[->, dashed,blue,looseness=2] (-0.901,-0.8849) to [bend left=80] (-1.045,-0.8849);

  \end{scope}
  
\end{tikzpicture}
\caption{Visco-Energetic solution of a double-well potential energy with an oscillating external loading and a quadratic viscous-correction $\delta(u,v)$, turned by a parameter $\mu>-\min W''$.}
\label{fig:6}
\end{figure}

This energy clearly satisfies \eqref{eq:W'-assumption}. Notice also that $W'(u)=u^3-u$ and $\min W''=-1$. Therefore, if we choose $\delta(u,v):=(v-u)^2$, according to Proposition \ref{prop:1}, $\Wird\delta=\Wsld\delta=W'$ and we expect a similar behaviour to BV solutions, with the optimal transition similar in the form to a ``double chain'' at every jump point.

If the loading is oscillating, for example $\ell(t)=\sin(t)$, $\alpha_\pm=\frac{1}{2}$ and we choose the initial datum such that $W'(u(a))=\ell(t)-\alpha_+$, the result is a loop typical of the hysteresis fenomena: the solution $u$ is locally constant when $\ell$ change direction.

\section{Visco-Energetic solutions with monotone loadings} \label{sec:4}
Visco-Energetic solutions of rate-independent systems in $\R$, driven by monotone loadings, involve the notion of the \textit{upper and lower monotone} (i.e. nondecreasing) \textit{envelopes} of the graph of $\Wird\delta$ and $\Wsld\delta$.
\par 
In this section we first focus on a few properties of this maps and their inverse and then we exhibit the explicit formulae characterizing Visco-Energetic solutions when $\ell$ is increasing or decreasing.

\subsection{Monotone envelopes of one-sided global slopes} \label{subsec:monotone-envelopes} 
\begin{definition}[Upper monotone envelope of $\Wird\delta$] For every $\bar{u}$ in $\R$, we define the maximal monotone map $\maxW(\cdot):\R\rightarrow\R$ 
\begin{equation} \begin{split}
\maxW(u):=\max_{\bar{u}\le v\le u}\Wird\delta(v)&\quad\text{if $u>\bar{u}$},\qquad \maxW(\bar{u}):=(-\infty,\Wird\delta(\bar{u})], \\
&\maxW(u)=\emptyset\quad \text{if $u<\bar{u}$}.
\end{split} \end{equation}
\end{definition}
We call $\maxW(\cdot)$ the \textit{upper monotone envelope of $\Wird\delta$} in the interval $(\bar{u},+\infty)$. The \emph{contact set} is defined by
\[
C^{\bar{u}}:=\{\bar{u}\}\cup\{u>\bar{u}:\Wird\delta(u)=\maxW(u)\}.
\] 
Thanks to \eqref{eq:W'-assumption}, it is easy to check that
\begin{equation}
\lim_{v\rightarrow -\infty}\Wird\delta(v)=-\infty,\qquad \lim_{v\rightarrow +\infty}\Wird\delta(v)=+\infty,
\end{equation}
so that the map $\maxW(\cdot)$ is monotone and surjective; it is also single-valued on $(\bar{u},+\infty)$ (where we identify the set $\maxW(u)$ with its unique element with a slight abuse of notation). We can thus consider the inverse graph $\invmaxW(\cdot):\R\rightarrow[\bar{u},+\infty)$ of $\maxW(\cdot)$: it is defined by
\[
u\in\invmaxW(\ell)\quad\Leftrightarrow\quad\ell\in\maxW(u)\qquad\text{for $u$, $\ell\in\R$}.
\]
Clearly, $\invmaxW(\cdot)$ is a maximal monotone graph in $\R$ and it is uniquely characterized by a left-continuous monotone function $p_{\sfl,\delta}^{\bar{u}}(\cdot)$ and a right-continuous monotone function $p_{\sfr,\delta}^{\bar{u}}(\cdot)$ such that
\[
\invmaxW(\ell)=[p_{\sfl,\delta}^{\bar{u}}(\ell),p_{\sfr,\delta}^{\bar{u}}(\ell)],\quad \text{i.e.}\quad \ell\in\maxW(u)\quad\Leftrightarrow\quad p_{\sfl,\delta}^{\bar{u}}(\ell)\le u \le p_{\sfr,\delta}^{\bar{u}}(\ell).
\]
We also consider a further selection in the graph of $\invmaxW(\cdot)$:
\[
\invmaxWc(\ell):=\{u\in \invmaxW(\ell):\Wird\delta(u)=\ell\}=\{u\in C^{\bar{u}}:\maxW(u)\ni\ell\}=\invmaxW(\ell)\cap C^{\bar{u}}.
\]
By introducing the set
\[
A_\delta^{\bar{u}}:=\{f:(\bar{u},+\infty)\rightarrow\R:\text{ $f$ is nondecreasing and fulfills $f\ge \Wird\delta$}\},
\]
we have
\begin{equation} \label{eq:13}
\maxW(\cdot)\restr{(\bar{u},+\infty)}\in A_\delta^{\bar{u}},\quad \Wird\delta(u)\le \maxW(u)\le f(u) \text{ for all $f\in A_\delta^{\bar{u}}, u\in(\bar{u},+\infty)$},
\end{equation}
so that $\maxW$ is the minimal nondecreasing map above the graph of $\Wird\delta$ in $(\bar{u},+\infty)$.
It immediately follows from \eqref{eq:13} that
\[
\maxW(u)=\inf\{f(u): f\in A_\delta^{\bar{u}}\}\qquad \text{for all $u>\bar{u}$}.
\]
The following result collects some simple properties of $p_{\sfl,\delta}^{\bar{u}}(\cdot)$ and $p_{\sfr,\delta}^{\bar{u}}(\cdot)$.
\begin{proposition} Assume \eqref{eq:W'-assumption}. Then for every $\ell\ge \Wird\delta(\bar{u})$ there holds
\begin{equation} \label{eq:14}
\Wird\delta(u)\le \ell\qquad \text{if $u\in [\bar{u},p_{\sfr,\delta}^{\bar{u}}(\ell)]$}.
\end{equation}
Moreover, for every $\ell\in\R$ we have
\begin{equation} \label{eq:15}
p_{\sfl,\delta}^{\bar{u}}(\ell)=\min \{u\ge \bar{u}:\Wird\delta(u)\ge\ell\},\qquad p_{\sfr,\delta}^{\bar{u}}(\ell)=\inf\{u\ge \bar{u}: \Wird\delta(u)>\ell\}.
\end{equation}
\end{proposition}

\begin{proof}
Property \eqref{eq:14} is an immediate consequence of the inequality $\Wird\delta\le \maxW(\cdot)$ in $[\bar{u},+\infty]$. 
\newline To prove the first of \eqref{eq:15} it is sufficient to notice that 
\[
\Wird\delta(u)\le \maxW(u)<\ell \quad \text{if }\bar{u}\le u< p_{\sfl,\delta}^{\bar{u}}(\ell),
\]
and $\maxW(u)=\ell$ if $u=p_{\sfl,\delta}^{\bar{u}}(\ell)$. For the second of \eqref{eq:15}, we observe that, when $u>p_{\sfr,\delta}^{\bar{u}}(\ell)$ we have $\maxW(u)>\ell$, and we know that there exists $v\in [p_{\sfr,\delta}^{\bar{u}}(\ell),u]$ such that $\Wird\delta(v)>\ell$. Since $u$ is arbitrary we get 
\[
p_{\sfl,\delta}^{\bar{u}}(\ell)\ge \inf\{u\ge \bar{u}: \Wird\delta(u)>\ell\}.
\]
The converse inequality follows from \eqref{eq:14}.
\end{proof}

\par
In a completely similar way we can introduce the \textit{maximal monotone map} below the graph of $\Wsld\delta$ on the interval $(-\infty,\bar{u}]$.
\begin{definition}[Lower monotone envelope of $\Wsld\delta$] For every $\bar{u}$ in $\R$, we define the maximal monotone map $\minW(\cdot):\R\rightarrow\R$ 
\begin{equation} \begin{split}
\minW(u):=\inf_{\bar{u}\le v\le u}\Wsld\delta(v)&\quad\text{if $u<\bar{u}$},\qquad \minW(\bar{u}):=[\Wsld\delta(\bar{u}),+\infty), \\
&\minW(u)=\emptyset\quad \text{if $u>\bar{u}$},
\end{split} \end{equation}
\end{definition}
which satisfies
\[
\minW(u)=\sup \{f(u):f\in B_\delta^{\bar{u}}\}\quad\text{for $u<\bar{u}$},
\]
where
\[
B_\delta^{\bar{u}}:=\{f:(-\infty,\bar{u})\rightarrow\R: \text{ $f$ is nondrecreasing and fulfills $f\le \Wsld\delta$} \}.
\]
As before, the inverse graph $\invminW(\cdot):=(\minW(\cdot))^{-1}:\R\rightarrow (-\infty,\bar{u}]$ can be represented as $\invminW(u)=[q_{\sfl,\delta}^{\bar{u}}(u),q_{\sfr,\delta}^{\bar{u}}(u)]$, where
\[
q_{\sfl,\delta}^{\bar{u}}(u)=\sup \{u\le \bar{u}: \Wsld\delta(u)<\ell\},\qquad q_{\sfr,\delta}^{\bar{u}}(u)=\max \{u\le \bar{u}: \Wsld\delta(u)\le\ell\}
\]
and we set
\[
\invminWc(\ell):=\{u\in \invminW(\ell):\Wsld\delta(u)=\ell\}.
\]

\subsection{Monotone loadings and Visco-Energetic solutions}
We apply the notions introduced the previous section to characterize Visco-Energetic solutions when $\ell$ is monotone. First of all, we provide an explicit formula yielding Visco-Energetic solutions for an increasing loading $\ell$. The case of a decreasing and of a piecewise monotone loading can be proved in a similar way.

\begin{theorem} \label{thm:monotonicity-theorem1} Let $\bar{u}\in \R$,  $\ell\in \rmC^1([a,b])$ be a nondecreasing loading such that
\begin{equation} \label{eq:monotone-loading-assumption1}
\ell(a)\ge \Wsld\delta(\bar{u})-\alpha_-.
\end{equation}
Any nondecreasing map $u:[a,b]\rightarrow\R$, with $u(a)=\bar{u}$,  such that for every $t\in(a,b]$ 
\begin{equation}\label{eq:monotone-loading-assumption2} 
\Wsld\delta(u(t))-\alpha_-\le W'(u(t)),\qquad u(t)\in \invmaxWc(\ell(t)-\alpha_+)
\end{equation}
is a Visco-Energetic solution of the rate-independent system $(\R,\cE,\Psi,\delta)$. In particular, \eqref{eq:monotone-loading-assumption2} yields
\begin{equation}\label{eq:monotone-loading-characterization1}
u(t)\in [p_{\sfl,\delta}^{\bar{u}}(\ell(t)-\alpha_+), p_{\sfr,\delta}^{\bar{u}}(\ell(t)-\alpha_+)]\quad\text{for every $t\in (a,b]$}.
\end{equation}
\end{theorem}

\begin{proof}
We apply Theorem \ref{thm:main-theorem}. Concerning the global stability condition, notice that \eqref{eq:monotone-loading-assumption2} yield
\begin{equation} \label{eq:31}
\Wird\delta(u(t))=W'(u(t))\qquad\text{for every $t\in(a,b]$}.
\end{equation}
Indeed, if $\Wird\delta(u(t))\neq W'(u(t))$, from Proposition \ref{prop: Wir.prop}, $\Wird\delta$ is decreasing in a neighborhood of $u(t)$, which contradicts the second of \eqref{eq:monotone-loading-assumption2}. Therefore, the first of \eqref{eq:monotone-loading-assumption2}, combined with \eqref{eq:31} gives \eqref{eq:1d-stability} for every $t\in(a,b]$.
\par
To check the equation \eqref{eq:equation1}, we set
\[
\gamma:=\inf \{t>a: u(t)>u(a)\}.
\]
If $\Wird\delta(u(a))<\ell(a)-\alpha_+$, then from \eqref{eq:monotone-loading-assumption2} $a\in \Ju$ and $\Wird\delta(\ur(a))=\ell(a)-\alpha_+$. Otherwise, $u(a)$ satisfies the stability condition and $u$ is clearly a constant Visco-Energetic solution on $[a,\gamma]$. Thus, it is not restrictive to assume that $\gamma=a$ by Proposition \ref{lem:1}.
In this case $\ur(t)>u(a)$ for every $t>a$ and by continuity \eqref{eq:monotone-loading-assumption2} yields
\begin{equation} \label{eq:32}
\Wird\delta(\ul(t))=\Wird\delta(\ur(t))=\ell(t)-\alpha_+\qquad \text{for every $t\in (a,b)$},
\end{equation}
and the second identity still holds in $t=a$. Thus, from the first of \eqref{eq:31} and the continuity of $W$ we finally get \eqref{eq:equation1}.

\par
To check the jump conditions, let us first notice that combining equation \eqref{eq:32} and \eqref{eq:monotone-loading-assumption2} 
\[
\Wird\delta(v)\le \ell(t)-\alpha_+\qquad\text{ for every $t\in (a,b]$, $\bar{u}\le v\le \ur(t)$}.
\] 
Then, from \eqref{eq:15} and the monotonicity of $u$ we get
\begin{equation} \label{eq:35}
p_{\sfl,\delta}^{\bar{u}}(\ell(t)-\alpha_+)\le \ul(t)\le u(t)\le \ur(t)\le p_{\sfr,\delta}^{\bar{u}}(\ell(t)-\alpha_+)\quad \text{for every $t\in(a,b]$}
\end{equation}
which yields \eqref{eq:jump1} and \eqref{eq:jump2} by \eqref{eq:14}. Moreover, the inequalities in \eqref{eq:35} also shows that \eqref{eq:monotone-loading-assumption2} implies \eqref{eq:monotone-loading-characterization1} 
\end{proof}

\par If the loading is decreasing, a similar result still holds. It can be proved with an adaptation of the proof of Theorem \ref{thm:monotonicity-theorem1}.
\begin{theorem} \label{thm:45}
 Let $\bar{u}\in \R$ and $\ell\in \rmC^1[a,b])$ be a nonincreasing loading such that
\begin{equation}
\ell(a)\le \Wird\delta(\bar{u})+\alpha_-.
\end{equation}
Any nonincreasing map $u:[a,b]\rightarrow\R$, with $u(a)=\bar{u}$, such that
\begin{equation} \label{eq:34}
\Wird\delta(u(t))+\alpha_+\ge W'(u(t)),\qquad u(t)\in \invminWc(\ell(t)+\alpha_-)\text{ for every $t\in (a,b]$}
\end{equation}
is a Visco-Energetic solution of the rate-independent system $(\R,\cE,\Psi,\delta)$. In particular, \eqref{eq:34} yields
\begin{equation}
u(t)\in [q_{\sfl,\delta}^{\bar{u}}(\ell(t)+\alpha_-), q_{\sfr,\delta}^{\bar{u}}(\ell(t)+\alpha_-)]\quad\text{for every $t\in [a,b]$},
\end{equation}
\end{theorem}

\begin{remark} The first condition of \eqref{eq:monotone-loading-assumption2} (resp. the first of \eqref{eq:34}) holds if the energy density $W$ satisfies the  $\delta$-convexity assumption \eqref{eq:41}. In this case
\[
\Wird\delta(u)=W'(u)=\Wsld\delta(u)\qquad\text{for every $u\in\R$}.
\]
In particular, it is satisfied if $W$ is $\lambda$-convex, see \eqref{eq:lambda-convex}, and we choose a quadratic $\delta$, tuned by a parameter $\mu\ge\min\{-\lambda,0\}$.
\end{remark}

\par
The next result shows that, under a slightly stronger condition on the initial data, any Visco-Energetic solution driven by an increasing loading admits a similar representation to \eqref{eq:monotone-loading-assumption2}: the second inclusion holds for every $t\not\in \Ju$.
\begin{theorem} [Nondecreasing loading] \label{thm:monotonicity-main-theorem}
Let $\ell\in \rmC^1([a,b])$ be a nondecreasing loading and let $u\in \BV([a,b],\R)$ be a Visco-Energetic solution of the rate-independent system $(\R,\cE,\Psi,\delta)$ satisfying
\begin{gather}
\ell(a)\ge \Wsld\delta(u(a))-\alpha_-, \tag{IC1} \\
W'<W'(u(a)) \text{ in a left neighborhood of u(a)}\quad \text{if $W'(u(a))=\ell(a)+\alpha_-$}; \tag{IC2} \label{eq:monotonicity-theorem-assumption1}
\end{gather}
and, for every $z<u(a)$,
\begin{equation}
\frac{W(z)-W(u(a))+\delta(u(a),z)}{z-u(a)}<\Wsld\delta(u(a)) \quad \text{if $\Wsld\delta(u(a))=\ell(a)+\alpha_-$}.\tag{IC3} \label{eq:monotonicity-theorem-assumpion2}
\end{equation}
Then, similarly to Theorem \ref{thm:monotonicity-theorem1}, $u$ satisfies
\begin{equation} \label{eq:monotonicity-theorem-thesis1}
u\text{ is nondecreasing},\qquad u(t)\in \invmaxWca(\ell(t)-\alpha_+)\quad\text{for every $t\in [a,b]\setminus \Ju$}
\end{equation}
and therefore
\begin{equation} \label{eq:monotonicity-theorem-thesis2}
u(t)\in [p_{\sfl,\delta}^{u(a)}(\ell(t)-\alpha_+), p_{\sfr,\delta}^{u(a)}(\ell(t)-\alpha_+)]\quad\text{for every $t\in [a,b]$}.
\end{equation}
\end{theorem}
This result is a generalization to the visco-energetic framework of \cite[Theorem 6.3]{Rossi-Savare13}. One of the technical point there is to avoid the extra assumption
\begin{equation} \label{eq:monotonicity-extra-assumption}
\Wsl(u)-\alpha_- \le \ell(t)< \Wir(u)+\alpha_+\quad \text{for every $u \in\R$},
\end{equation}
which is not immediately satisfied if $\alpha_\pm$ are very small. In our context, if $\delta$ is too small \eqref{eq:monotonicity-extra-assumption} is still not satisfied even if we replace the one-sided slopes with their $\delta$-corrected versions.  
\par
The next technical lemma contributes to solve this issue. Compared with the same result in the energetic setting, we need a more refined analysis of the behaviour of $u$ at jumps.
\begin{lemma} \label{lem:monotonicity-proof-lemma} Under the same assumptions of Theorem \ref{thm:monotonicity-main-theorem}, let $a<\sigma'<\sigma\le b$ be such that
\begin{equation}
\ell(t)-W'(\ur(t))>-\alpha_-=\ell(\sigma)-W'(\ur(\sigma))\quad\text{for every $t\in[\sigma',\sigma)$}.
\end{equation}
Then $\sigma\notin \Ju$.
\end{lemma}

\begin{proof}
We argue by contradiction and assume that $\sigma\in \Ju$. In view of \eqref{eq:equation1} necessarily
\begin{equation} \label{eq:monotone-loading-ineq2}
\ul(\sigma)>\ur(\sigma),
\end{equation}
and \eqref{eq:equation2} shows that $\ur$ is nondecreasing in $[\sigma',\sigma)$. Moreover, combining \eqref{eq:1d-stabilityleft} and \eqref{eq:jump2},
\[
\ell(\sigma)+\alpha_-=\Wsld\delta(\ul(\sigma))>W'(\ul(\sigma)),
\]
so that by Proposition \ref{prop: Wir.prop} there exist $\tilde{u}<\ul(\sigma)$ which attains the supremum in the definition of $\Wsld\delta(\ul(\sigma))$ and a neighborhood of $\ul(\sigma)$ in which $\Wsld\delta(u)$ is decreasing. 
\par

We want to prove that $\ul(\sigma)=u(a)$. We consider the set
\[
\cP:=\{\rho\in [a,\sigma): \ur(t)\equiv \ul(\sigma),\mbox{ }\ell(t)=\ell(\sigma)\quad \text{for all $t\in[\rho,\sigma)$}\},
\]
and we prove that $[a,\sigma)=\cP$.

\underline{Claim 1}. $\cP\neq \emptyset$ and $\cP$ is closed in $[a,\sigma)$.
\newline We need to show that $\ell$ and $u$ are constant in a left neighborhood of $\sigma$. We already know that they are nondecreasing in $[\sigma',\sigma)$. To show that they are also nonincreasing we argue by contradiction: assume that there exists a sequence $t_n<\sigma$ converging to $\sigma$ such that 
\[
u_n:=\ur(t_n)\uparrow \ul(\sigma), \mbox{ }\ell_n:=\ell(t_n)\uparrow\ell(\sigma) \quad\text{and $u_n+\ell_n<\ul(\sigma)+\ell(\sigma)$}.
\]
Then, for $n$ great enough, $u_n>\tilde{u}$. If $u_n<\ul(\sigma)$ the global stability \eqref{eq:1d-stabilityright} and the jump condition \eqref{eq:17} yield
\[
\ell_n+\alpha_-\ge \Wsld\delta(u_n)>\ell(\sigma)+\alpha_-\ge \ell_n+\alpha_-,
\]
which is absurd. Similarly, if $\ell_n<\ell(\sigma)$,
\[
\ell_n+\alpha_-\ge \Wsld\delta(u_n)\ge\ell(\sigma)+\alpha_-> \ell_n+\alpha_-.
\]
This proves that $\cP$ contains a left neighborhood of $\sigma$ and then it is non-empty. Moreover, $\cP$ is clearly closed in $[a,\sigma)$.
\par
\underline{Claim 2}. Suppose that $\rho\in \cP$. Then $\rho\notin \Ju$.
\newline
By contradiction suppose that $\rho\in \Ju$. From Proposition \ref{prop: Wir.prop}, $\Wsl$ is decreasing in an open set containing $\ul(\sigma))=\ur(\rho)$. From \eqref{eq:jump2} the only possibility is that $\ul(\rho)<\ur(\rho)$. 
\newline
Suppose that $\ul(\rho)\le \ur(\sigma)$. Then we consider an optimal transition $\vartheta\in \rmC(E,\R)$ connecting $\ul(\rho)$ and $\ur(\rho)$.
Clearly, $\ur(\sigma)\notin E$ since $\cE(\rho,\ur(\sigma))=\cE(\sigma,\ur(\sigma))$ but from $\rm{(J_{VE})}$ the energy is decreasing during a transition:
\[
\cE(\sigma,\ur(\sigma))<\cE(\sigma,\ul(\sigma))=\cE(\rho,\ur(\rho))<\cE(\rho,\ur(\sigma)).
\]
Then $\ur(\sigma)$ must be in a hole of $E$. Combing Theorem \ref{prop:2} and  Proposition \ref{prop:3},
\[
\Wird\delta(\ur(\sigma))<\ell(\sigma)-\alpha_+,
\]
which contradicts the global stability \eqref{eq:1d-stabilityright}. In a similar way we can discuss the case $\ul(\rho)\ge \ur(\sigma)$.
\par
\underline{Claim 3}. If $a<\rho\in \cP$, there exist $\varepsilon>0$ such that $\ell(t)\equiv \ell(\rho)\equiv \ell(\sigma)$ and  $u(t)=u(\rho)=\ul(\sigma)$ for every $t\in [\rho-\varepsilon,\rho]$.
\newline Thanks to Claim 2, $\rho\notin \Ju$. Then we can argue as in Claim 1, starting from $\ur(\rho)=\ul(\sigma)$ and $\ell(\rho)=\ell(\sigma)$.
\par
\underline{Conclusion}. $\cP$ is also open in $[a,\sigma)$ since for every $\rho\in \cP\cap (a,\sigma)$, it contains a left neighborhood of $\rho$ ($\cP$ obviously contains also a right neighborhood of $\rho$). Since $\cP$ is both open and closed, $\cP=[a.\sigma)$.  
\par
Another application of Claim 2, combined with \eqref{eq:monotonicity-theorem-assumpion2}, which prevents the case $u(a)>\ur(a)$, yields that $a\notin \Ju$, so that $\ul(\sigma)=u(a)$ and $\ell(\sigma)=\ell(a)$. Finally, another application of \eqref{eq:monotonicity-theorem-assumpion2} gives the contradiction with \eqref{eq:monotone-loading-ineq2}.
\end{proof}

With these notions at our disposal, the proof of Theorem \ref{thm:monotonicity-main-theorem} is a simple adaptation of \cite[Theorem 6.3]{Rossi-Savare13}. For completeness, we report the steps in details.
\begin{proof}[Proof of Theorem \ref{thm:monotonicity-main-theorem}] We split again the argument in various steps.
\par
\underline{Claim 1}. There exists $\gamma\in [a,b]$ such that $\ell(t)-W'(\ur(r))>\gamma$ for all $t\in (\gamma,b]$ and $u(t)\equiv u(a),\,\ell(t)\equiv\ell(a)$ in $[a,\gamma]$.
\newline 
Let us consider the set 
\[
\Sigma:=\{t\in [a,b]: W'(\ur(t))=\ell(t)+\alpha_-\}
\]
and observe that $t_n\in\Sigma,\quad t_n\downarrow t\quad \Rightarrow\quad t\in \Sigma$.
If $a\in \Sigma$, we denote by $\Sigma_a$ the connected component of $\Sigma$ containing $a$ and we set $\gamma:=\sup \Sigma_a$. If $\gamma>a$, then
\[
W'(\ur(t))=\ell(t)+\alpha_-\quad\text{for every $t\in[a,\gamma]$},
\]
so that by \eqref{eq:equation1} $u$ is nonincreasing in $[a,\gamma]$. Assumption \eqref{eq:monotonicity-theorem-assumpion2} imply that $a\notin \Ju$ and $\ur(a)=u(a)$. Since also $\ell$ is nondecreasing we conclude by \eqref{eq:monotonicity-theorem-assumption1} and \eqref{eq:equation2} that $u(t)\equiv u(a)$ and $\ell(t)\equiv\ell(a)$ in $[a,\gamma]$; moreover, by the same argument, $\gamma\notin \Ju$, so that $\gamma\in \Sigma$. When $a\notin \Sigma$ we simply set $\gamma:=a$ and $\Sigma_a=\emptyset$.
\par
The claim then follows if we show that $\Sigma\setminus\Sigma_a$ is empty. This is trivial if $\gamma=b$. If $\gamma<b$ we suppose $\Sigma\setminus \Sigma_a \neq \emptyset$ and we argue by contradiction. We can find points $\gamma_2>\gamma_1>\gamma$ such that $\gamma_1\notin \Sigma$ and $\gamma_2\in \Sigma$. We can consider $\sigma:=\min(\Sigma\cap [\gamma_1,b])>\gamma_1>\gamma$. Lemma \ref{lem:monotonicity-proof-lemma} with $\sigma':=\gamma_1$ yields that $\sigma\notin \Ju$, so that we can find $\varepsilon>0$ such that
\begin{equation} \label{eq:monotone-loading-ineq3}
-\alpha_-<\ell(t)-W'(\ur(t))<\alpha_+\qquad\text{for every $t\in(\sigma-\varepsilon,\sigma)$}.
\end{equation}
Point $b)$ of Theorem \ref{thm:main-theorem} implies that $\ur(t)=u(t)\equiv \ul(\sigma)=\ur(\sigma)$ is constant in $(\sigma-\varepsilon,\sigma)$. Hence, $W'(\ur(t))\equiv W'(\ur(\sigma))=\ell(\sigma)+\delta_-\ge \ell(t)+\delta_-$ for every $t\in (\sigma-\varepsilon,\sigma)$, since $\ell$ is nondecreasing and $\sigma\in \Sigma$. This contradicts \eqref{eq:monotone-loading-ineq3}.
\par
\underline{Claim 2}. $u$ is nondecreasing in $[a,b]$.
\newline
Relation \eqref{eq:equation2} and Claim 1 imply that $\left( u'\right)^-\left([a,b)\right)=0$, so that $u$ in nondrecreasing in $[a,b)$. If $b$ is a jump point, then by \eqref{eq:jump1} $u(b)>\ul(b)$.

\par
\underline{Claim 3}. Let $B:=\{t\in [\gamma,b]:\Wird\delta(u(a))+\alpha_+=\ell(t)\}$ and let $\beta:=\min B$ (with the convention $\beta=b$ if $B$ is empty). Then $u(t)\equiv u(a)$ in $[a,\beta)$ and
\begin{equation} \label{eq:monotone-loading-identity1}
\Wird\delta(\ul(t))=\Wird\delta(\ur(t))=\ell(t)-\alpha_+\quad\text{for all $t\in (\beta,b)$}.
\end{equation}
In particular, $\ul(t)\ge p_{\sfl,\delta}^{u(a)}(\ell(t)-\alpha_+)$ for all $t\in [a,b]$.
\newline The first statement follows from the previous Claim and Corollary \ref{cor:main-theorem-corollary}. 
\newline To prove the second identity in \eqref{eq:monotone-loading-identity1} for $\ur(t)$, we argue by contradiction and we suppose that exists a point $s\in(\beta,b]$ such that $\Wird\delta(\ur(s))+\alpha_+>\ell(s)$. Then, in view of Corollary \eqref{cor:main-theorem-corollary} $u$ is locally constant around $s$. Since $\ell$ is nondecreasing, because of $\eqref{eq:equation1}$, we conclude that $u(t)\equiv u(s)$ for every $t \in[\gamma,s]$, so that $s\le \beta$, a contradiction. The first identity of \eqref{eq:monotone-loading-identity1} follows by continuity and by \eqref{eq:equation1}.
\newline The last statement is a consequence of \eqref{eq:15}. Notice that we can also take $t=b$ since $\Wird\delta(\ul(t))=\ell(t)-\alpha_+$ still holds in $t=b$.
\par
\underline{Claim 4}. For all $t\in [a,b]$ we have $\ur(t)\le p_{\sfr,\delta}^{u(a)}(\ell(t)-\alpha_+)$.
\newline 
If $\ur(t)=u(a)$ there is nothing to prove. Otherwise, let $t\ge\beta$ and take $z\in (u(a),\ur(t))$. Since $u$ is nondecreasing, there exists $s\in [\beta,t]$ such that $\ul(s)\le z\le \ur(s)$, so that \eqref{eq:jump2} (in the case $s\in\Ju$) or \eqref{eq:equation1} (in the case $\ul(s)=\ur(s)$) yield
\[
\Wird\delta(z)\le\ell(s)-\alpha_+\le \ell(t)-\alpha_+,
\] 
since $\ell$ is nondecreasing. Being $z<\ur(t)$ arbitrary, the claim follows from the second of \eqref{eq:15}.
\par
\underline{Conclusion}. Combining Claim 2, Claim 3 and Claim 4, we get
\[
p_{\sfl,\delta}^{u(a)}(\ell(t)-\alpha_+)\le \ul(t)\le u(t)\le\ur(t) \le p_{\sfr,\delta}^{u(a)}(\ell(t)-\alpha_+)\quad \text{for every $t\in [a,b]$},
\]
which proves relation \eqref{eq:monotonicity-theorem-thesis2}. Finally, \eqref{eq:monotonicity-theorem-thesis1} is due to \eqref{eq:monotonicity-theorem-thesis2} and \eqref{eq:monotone-loading-identity1}.
\end{proof}

In a similar way, we can deduce the the characterization of Visco-Energetic solutions in the case of a decreasing load.
\begin{theorem}[Nonincreasing loading] Let $\ell\in \rmC^1([a,b])$ be a nonincreasing loading and let $u\in \rm{BV} ([a,b],\R)$ be a Visco-Energetic solution of the rate-independent system $(\R,\cE,\Psi,\delta)$ satisfying
\begin{gather}
\ell(a)\le \Wird\delta(u(a))+\alpha_+,  \\
W'>W'(u(a)) \text{ in a right neighborhood of u(a)}\quad \text{if $W'(u(a))=\ell(a)-\alpha_+$};
\end{gather}
and, for every $z>u(a)$,
\begin{equation}
\frac{W(z)-W(u(a))+\delta(u(a),z)}{z-u(a)}>\Wird\delta(u(a)) \quad \text{if $\Wird\delta(u(a))=\ell(a)-\alpha_+$}.
\end{equation}
Then, similarly to Theorem \ref{thm:45}, $u$ satisfies
\begin{equation}
u\text{ is nonincreasing},\qquad u(t)\in \invminWca(\ell(t)+\alpha_-)\quad\text{for every $t\in [a,b]\setminus \Ju$}
\end{equation}
and therefore
\begin{equation}
u(t)\in [q_{\sfl,\delta}^{u(a)}(\ell(t)\alpha_-), q_{\sfr,\delta}^{u(a)}(\ell(t)+\alpha_-)]\quad\text{for every $t\in [a,b]$}.
\end{equation}

\end{theorem}

\paragraph{Example.} We conclude with a final example, involving a more complex potential $W$ (see figure \ref{fig:7}). When $W\in \rmC^2([a,b];\R)$ and we choose $\delta(u,v):=\frac{\mu}{2}(v-u)^2$, with $\mu\ge-\min W''$, Visco-Energetic solutions follow the monotone envelope of $W'+\alpha_+$. 

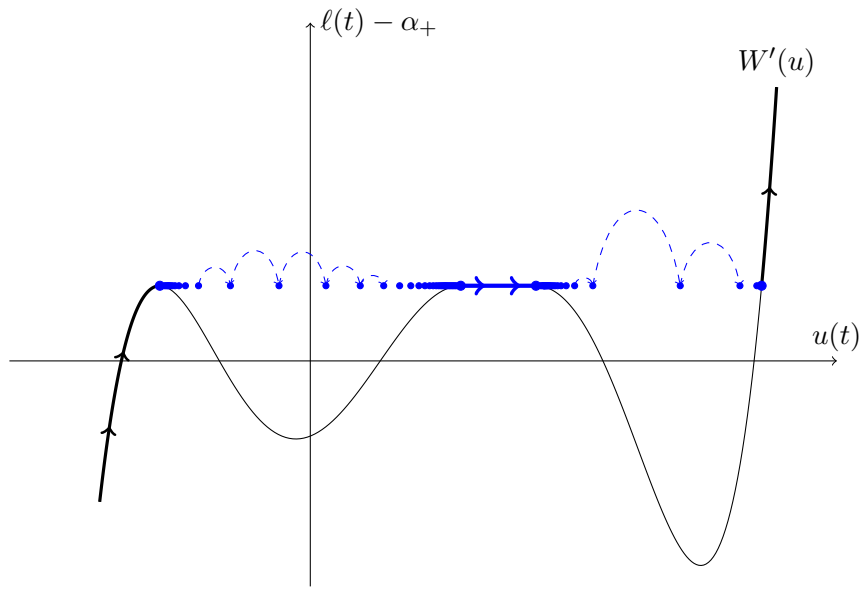
\begin{figure} [H]
\centering
\begin{tikzpicture}

\draw[->] (-4,0) -- (7,0) node[above] {$u(t)$};
\draw[->] (0,-3) -- (0,4.5) node[right] {$\ell(t)-\alpha_+$};

 \begin{scope}[scale=2]
    \draw[domain=-1.4:1,samples=100] plot ({\x}, {((\x)^2-1)^2*(2/5*\x-1)+0.5});
    \draw[domain=1.5:3.1,samples=100] plot ({\x}, {((\x-0.5)^2-1)^2*(2/5*(\x-0.5)-1)+0.5}) node[above] {$W'(u)$};
    \draw[domain=1:1.5,ultra thick,samples=100, blue, arrow inside={}{0.33,0.80}] plot ({\x}, {0.5});
    \draw[domain=-1.4:-1, very thick,samples=100] plot ({\x}, {((\x)^2-1)^2*(2/5*\x-1)+0.5}) [arrow inside={}{0.33,0.66}];
     \draw[domain=3:3.1,very thick,samples=100] plot ({\x}, {((\x-0.5)^2-1)^2*(2/5*(\x-0.5)-1)+0.5}) [arrow inside={}{0.50}];
    \foreach \Point in {(-1,0.5),(3,0.5)}{
    \draw[fill=blue, blue] \Point circle(0.03);
}
    \foreach \Point in {(1,0.5),(1.5,0.5)}{
    \draw[fill=blue, blue] \Point circle(0.03);
}
     \foreach \Point in {(-0.97,0.5),(-0.96,0.5),(-0.95,0.5),(-0.944,0.5), (-0.937,0.5), (-0.928,0.5), (-0.916,0.5), (-0.899,0.5), (-0.873,0.5), (-0.830,0.5), (-0.743,0.5), (-0.531,0.5) node[below] {\qquad $viscous$}, (-0.208,0.5), (0.104,0.5),(0.332,0.5),(0.488,0.5),(0.594,0.5),(0.668,0.5),(0.722,0.5),(0.762,0.5),(0.793,0.5),(0.817,0.5),(0.837,0.5),(0.853,0.5),(0.866,0.5),(0.877,0.5),(0.887,0.5),(0.895,0.5),(0.902,0.5),(0.91,0.5),(0.92,0.5),(0.93,0.5),(0.94,0.5),(0.95,0.5),(0.96,0.5),(0.97,0.5)}{
    \draw[fill=blue, blue] \Point circle(0.02);
   
  }

\foreach \Point in {(1.53,0.5),(1.54,0.5),(1.55,0.5),(1.552,0.5),(1.554,0.5),(1.557,0.5),(1.56,0.5),(1.563,0.5),(1.567,0.5), (1.571,0.5), (1.576,0.5), (1.581,0.5), (1.587,0.5), (1.594,0.5) node[below left] {\qquad $sliding$}, (1.603,0.5), (1.614,0.5), (1.627,0.5), (1.644,0.5), (1.667,0.5),(1.701,0.5),(1.757,0.5),(1.878,0.5),(2.459,0.5) node[below] {$viscous$},(2.856,0.5),(2.968,0.5),(2.993,0.5)}{
    \draw[fill=blue, blue] \Point circle(0.02);
}

\draw[->, dashed,blue,looseness=2] (-0.743,0.5) to [bend left=80] (-0.531,0.5);
\draw[->, dashed,blue,looseness=2.5] (-0.531,0.5) to [bend left=80] (-0.208,0.5);
\draw[->, dashed,blue,looseness=2.5] (-0.208,0.5) to [bend left=80] (0.104,0.5);
\draw[->, dashed,blue,looseness=2] (0.104,0.5) to [bend left=80] (0.332,0.5);
\draw[->, dashed,blue,looseness=1.5] (0.332,0.5) to [bend left=80] (0.488,0.5);

\draw[->, dashed,blue,looseness=1.5] (1.757,0.5) to [bend left=80] (1.878,0.5);
\draw[->, dashed,blue,looseness=3] (1.878,0.5) to [bend left=80] (2.459,0.5);
\draw[->, dashed,blue,looseness=2.5] (2.459,0.5) to [bend left=80] (2.856,0.5);


  \end{scope}
  
\end{tikzpicture}
\caption{Visco-Energetic solutions of a nonconvex energy and an increasing loading. The optimal transition is a combination of sliding and viscous parts.}
\label{fig:7}
\end{figure}

\def\cprime{$'$} \def\cprime{$'$} \def\cprime{$'$}

\end{document}